\documentclass[12pt]{article}
\pdfoutput=1 
\usepackage[english]{babel}

\usepackage[letterpaper,top=2cm,bottom=2cm,left=3cm,right=3cm,marginparwidth=1.75cm]{geometry}

\usepackage{amsmath}
\usepackage{graphicx}
\usepackage[colorlinks=true, allcolors=blue]{hyperref}
\usepackage[utf8]{inputenc}
\usepackage{url}
\usepackage{hyperref}
\usepackage{amssymb}
\usepackage{accents}
\usepackage[english]{babel}
\usepackage{amsthm}
\usepackage{geometry,mathtools}
\usepackage{esvect}
\usepackage{dsfont}
\usepackage{caption}
\usepackage{subcaption}
\usepackage{tikz}
\usetikzlibrary{automata,arrows,positioning,calc}
\usepackage{blkarray}
\usepackage{authblk}
\usepackage{mathtools}
\mathtoolsset{showonlyrefs}

\usepackage{enumerate}
\usepackage{enumitem}   
\usetikzlibrary{shapes,arrows,graphs,graphs.standard,shapes.misc}
\usetikzlibrary{matrix,decorations.pathreplacing, calc, positioning,fit}
\usetikzlibrary{shapes.geometric}
\usepackage{xcolor}
\usetikzlibrary{matrix,decorations.pathreplacing}
\pgfkeys{tikz/mymatrixenv/.style={decoration=brace,every left delimiter/.style={xshift=4pt},every right delimiter/.style={xshift=-4pt}}}
\pgfkeys{tikz/mymatrix/.style={matrix of math nodes,left delimiter=[,right delimiter={]},inner sep=1pt,row sep=0em,column sep=0em,nodes={inner sep=6pt}}}

\newcommand{\Rset}{\mathbb{R}}
\newcommand{\Cval}{\mathbb{C}}
\newcommand{\Nset}{\mathbb{N}}

\newcommand*{\rom}[1]{\expandafter\@slowromancap\romannumeral #1@}

\newlist{steps}{enumerate}{1}
\setlist[steps, 1]{label = \textbf{Step \arabic*}:}
\newtheorem{thm}{Theorem}
\newtheorem{prop}{Proposition}
\newtheorem{defn}{Definition}[section]
\newtheorem{lem}{Lemma}

\newtheorem{cor}{Corollary}[thm] 

\newtheorem{rem}{Remark}
\theoremstyle{definition}
\newtheorem{exmp}{Example}[section]

\title{Vector-Valued Gossip over $w$-Holonomic Networks}

\author{Erkan Bayram\textsuperscript{*}, Mohamed-Ali Belabbas\textsuperscript{*}, Tamer Ba{\c{s}}ar\thanks{E. Bayram, M.-A. Belabbas and T. Ba{\c{s}}ar are with Coordinated Science Laboratory, University of Illinois Urbana-Champaign, Email: \texttt{\{ebayram2,belabbas,basar1\}@illinois.edu}}}
\date{}

\begin{document}
\maketitle
\begin{abstract}
We study the weighted average consensus problem for a gossip network of agents with vector-valued states. For a given matrix-weighted graph, the gossip process is described by a sequence of pairs of adjacent agents communicating and updating their states based on the edge matrix weight. 
Our key contribution is providing conditions for the convergence of this non-homogeneous Markov process as well as the characterization of its limit set.
To this end, we introduce the notion of ``$w$-holonomy'' of a set of stochastic matrices, which enables the characterization of sequences of gossiping pairs resulting in reaching a desired consensus in a decentralized manner. Stated otherwise, our result characterizes the limiting behavior of infinite products of (non-commuting, possibly with absorbing states) stochastic matrices.

\textbf{Keywords:} Consensus; Gossiping; Non-Homogeneous Markov Processes; Holonomy; Convergence of Matrix Products; Permutation Groups

\end{abstract}

\section{Introduction}\label{sec:intro}

Consensus entails reaching an agreement between a set of agents~\cite{degroot1974reaching}. Many applications of distributed control systems require agents to reach a consensus for a given quantity; for example consensus to the average value of their respective initial states. An extension of average value consensus is the weighted average consensus, in which each agent contributes to the agreed-upon consensus value based on its assigned weight;  see the literature review below for more details.

In this paper, we study the weighted average consensus problem for a gossiping network of agents with vector-valued states. Specifically, given a matrix weighted communication graph, we study the process whereby at each time step, two adjacent agents in the network communicate and update their states based on the matrix weight of the edge adjoining them. These two agents are called a gossiping pair and the overall process is called a \textbf{weighted gossip process}~\cite{boyd2006randomized,benezit2010weighted}. It is akin to a non-homogeneous Markov process, and the study of its convergence thus reduces to the study of convergence of an infinite product of row stochastic matrices taken from a finite set. It is well known that this is a hard problem for which no general solution is known. This is due in part to the fact that, save for particular cases such as a set of commuting matrices, the order in which stochastic matrices appear in the infinite product clearly affects the limit set; in fact, this set can be a continuum (see, e.g.~\cite{belabbas2021triangulated} for examples) or it can be finite (e.g.~\cite{chen2022gossip}). 

We adopt here a vantage point on the consensus problem similar to the one of~\cite{chen2022gossip}, where the notion of holonomy of a set of stochastic matrices was introduced. There, the authors used the term holonomy to indicate a change of a certain left eigenvector (referred to as {\em weight vector} below) corresponding to eigenvalue $1$ of the product of stochastic matrices along any cycle in the graph. For each cycle in the graph, one can associate a holonomy group (see~\cite{wolf2002dif} for the precise definition of a holonomy group). This group characterizes how the eigenvector changes as gossiping occurs along the cycle. In~\cite{chen2022gossip}, the authors consider gossip processes for agents with scalar states and impose that the entries of a gossip matrix be strictly positive. Together, these restrictions imply that the holonomy group for a cycle, if it exists, can only be the trivial group. 

Our work here extends this earlier work in two fundamental ways. First, we allow vector-valued states for the agents. Second, and more importantly, we allow zero entries in the gossip matrices. Said otherwise, we allow for update matrices that have absorbing states (i.e., have a standard unit vector as a row). These extensions together make possible the existence of a non-trivial, finite holonomy group in a gossip process, {a phenomenon that does not arise in classical models that only consider update matrices with strictly positive entries. The notion of holonomy in this context refers to the structure imposed by a set of stochastic matrices on a given weight vector. In particular, when these matrices exhibit finite orbit sets, meaning that iterating the update process results in a cyclic progression through a finite set of weight vectors, this provides a powerful algebraic tool for analyzing the long-term behavior of the system.

More intuitively, an orbit in this setting represents the sequence of weight vectors generated by successive applications of the state transition matrix around cycles. When the orbit is finite, the process exhibits a form of recurrence, which significantly influences the system’s convergence properties. By studying the holonomy structure of the update dynamics, we can determine whether the gossip process stabilizes to a predictable state, oscillates within a bounded region, or exhibits other structured asymptotic behaviors.
}

More generally, the hereby adopted set-up raises the following questions: 
\begin{enumerate}
    \item How to understand the appearance of non-trivial holonomy groups? I.e., situations where the weight vector changes after completing one loop in the gossip graph,  but then returns to its initial value after completing this loop a finite number of times? 
    \item Can we still guarantee the convergence of the weighted gossip process to a limit or a finite limit set by following a sequence of gossiping pairs in a decentralized manner? 
    \item How does the potential presence of absorbing states in gossip updates impact the consensus weight vector?  These three questions are fully addressed in this paper.
\end{enumerate}

To understand the phenomenon described in the first question, we introduce a concept which we call {\em $w$-holonomy} of a set of stochastic matrices. This concept helps us describe the set of stochastic matrices that possess finite orbit sets when acting on some vectors. Such matrices are the ones enabling the appearance of holonomy groups in gossip processes.

For the second question, we introduce the so-called derived graph of the (communication) graph $G$ for the weight vector $w$, which we denote by $D_G(w)$. Infinite exhaustive walks in the derived graph {(which visits every node in the graph infinitely often)} will correspond to {\em allowable sequences} of updates in the gossip process. These updates can be implemented in a decentralized manner, and yield a process which converges to a finite limit set. 

For the third question, the presence of zeros and ones in the update matrices can significantly impact the consensus weights in our analysis. In particular, they can lead to some agents not contributing to the consensus value average and to the appearance of permutation matrices as update matrices. In fact, even when none of the update matrices are permutation matrices, their product within the gossip iteration can result in a permutation matrix (as will be illustrated later). This fact greatly complicates the analysis and is at the root of the existence of finite limit sets.

{ We summarize our main contributions as follows:
\begin{itemize}
    \item We provide conditions for the convergence of a non-homogeneous Markov process in the presence of absorbing states and non-trivial holonomy.
    \item We prove that the corresponding infinite product of stochastic matrices, whose order is given by the allowable sequences, converges to a finite limit set.     
    \item We explicitly characterize the elements of this limit set. Our results show that the elements in the limit set can be relabeled as a block diagonal matrix, with each block either having identical rows or being a permutation block.
\end{itemize}

These findings are closely related to the concept of {\em multiple consensus} in the literature. Specifically, when there exists a fixed partition of the components of the vector states such that consensus occurs for each subvector corresponding to this partition, the process is said to exhibit {\em multiple consensus} (or class-ergodicity)~\cite{bolouki2015consensus,bolouki2013ergodicity,touri2012approximations} and the elements of this partition are referred to as a {\em consensus clusters}. In our results, each element in the limit set exhibits {\em multiple consensus} and the corresponding consensus clusters can be determined by the method we have described.}

The paper is organized as follows: We provide a brief review of the relevant literature on distributed control and weighted average consensus in the following paragraph. We then describe the notations and conventions used in the paper at the end of this section. In Section~\ref{sec:prelim}, we provide a precise formulation of the problem solved in this paper. The notion of holonomy and the main results of the paper are presented in Section~\ref{sec:main_resutls}. The proof of the main theorem is provided in Section~\ref{sec:proof} along with some auxiliary results. A summary of the results of the paper and outlook for future research are provided in Section~\ref{sec:conclusion}.

\textbf{Literature Review} In recent decades, there has been an increase in the applications of multi-agent systems and distributed control. These applications aim to achieve consensus among agents, as seen in works like \cite{zhu_etal2021federated,olfati2007distributed,chen_etal2017cluster,sundaram2018distributed,touri2015continuous,bayram2024age,bayram2024ageCoded,bolouki2015eminence}. Many of these systems involve agents with multiple states, highlighting the importance of addressing weighted average consensus.

The field of weighted average consensus has seen diverse perspectives and contributions over the years, such as \cite{zhu2010discrete,blondel2014decide,nedic2016convergence,tsitsiklis1984problems,kashyap2007quantized,etesami2015convergence} and \cite{chen2016distributed}. Research has tackled challenges like time delays and asynchronous information spread~\cite{li2021event,Morse_etal2008ReachingConsensus,fang2005asynchronous}, as well as changing network topologies due to link failures or reconfiguration~\cite{Morse_etal2008ReachingConsensus,Morse_etAl2008Dynamically,ren2005consensus}. Moreover, works presented by~\cite{hendrickx2011new,hendrickx2012convergence} have focused on continuous-time consensus problems. Various techniques have been used to solve consensus problems, including Lyapunov function-based methods~\cite{nedic2016convergence,fagnani2008randomized}, and approaches inspired by ergodicity theory~\cite{touri2012backward,touri2010ergodicity,touri2012approximations,aghajan2021ergodicity}. Furthermore, research efforts have addressed the constant network topology driven by the gossip process~\cite{belabbas2021triangulated,chen2022gossip,liu2011deterministic,he2011periodic,bayram2025construct}. Our work falls within the scope of this latter category of research.

\textbf{Notations and conventions.} We denote by $G=(V,E)$ an undirected graph, with $V=\{v_1,\ldots,v_{|V|}\}$ the node set and $E=\{e_1,\ldots,e_{|E|}\}$ the edge set. The edge linking nodes $v_i$ and $v_j$ is denoted by $(v_i, v_j)$, a self-arc or self loop is denoted by $(v_i , v_i)$. We call $G$ {\em simple}  if it has no self-loops. 

Given a sequence of edges $\gamma = e_1 \cdots  e_k$ in $E$, a node $v \in V$ is called \textbf{covered} by $\gamma$  if it is incident to an edge in $\gamma$. Given a sequence $\gamma = e_1 e_2 \cdots $, we say that $\gamma'$ is a string of $\gamma$ if it is a contiguous
subsequence, i.e., $\gamma'= e_k e_{k+1} \cdots e_l$ for some $k \geq 1$ and $l \geq k$. Let $\gamma = e_1 \cdots e_k$
be a finite sequence and $e_{k+1}$ be an edge of $G$. The sequence $e_1 \cdots e_k e_{k+1}$ obtained by adding $e_{k+1}$ to the end of $\gamma$ is denoted by $\gamma \lor e_{k+1}$.

For a given simple undirected graph $G$ as above, we denote by $\vec{G} = (V, \vec{E} )$ a directed graph on the same node set and with a ``bidirectionalized'' edge set; precisely, $\vec{E}$ is defined as follows: we assign to every edge $(v_i, v_j)$ of $G$, $i \neq j$, two directed edges $v_iv_j$ and $v_jv_i$. 

We denote a walk in $\vec{G}$ either by the succession of edges or the succession of nodes visited. We say that $\gamma = v_{i_1} \cdots v_{i_k}$ is a walk in the directed graph $\vec{G}$ if $v_{i_l} v_{i_{l+1}}$, for $\ell =1, \cdots , k-1$, is an edge of $\vec{G}$. 
We refer to $v_{i_1}$ and $v_{i_k}$ as the starting- and ending-nodes of $\gamma$, respectively. We define $\gamma^{-1} :=v_{i_k} v_{i_{k-1}}
\cdots v_{i_{1}}$. Let $\gamma'=v_{i_{l}} v_{i_{l+1}} \cdots v_{i_{m}}$ be another walk in $\vec{G}$. We denote by $\gamma \lor \gamma' := v_{i_1} \cdots v_{i_{k}}v_{i_{l}} \cdots v_{i_m}$ the concatenation of the two walks.

{A vector $p \in \mathbb{R}^n$ is a probability vector if $p_i \geq 0$ and $\sum_{i=1}^{n} p_i = 1$. The set of such vectors is the $(n-1)$-simplex $\Delta^{n-1}$. Its interior $\operatorname{int} \Delta^{n-1}$ in the Euclidean topology is one where all entries of $p$ are strictly positive.}

We let $\mathds{1}$ be a vector of all ones, whose dimension will be clear from the context.

\section{Preliminaries}\label{sec:prelim}
{\textbf{Gossip Process on a Matrix Weighted Graph.}} Let $G=(V,E)$ be an undirected simple graph on $n$ nodes. Each node represents an agent, and each agents' state is a vector in $\mathbb{R}^m$. We denote the state vector of the agent $i$ at time $t$ by ${x^i(t)}= \left[ x_1^i(t) , x_2^i(t) ,\ldots , x_m^i(t) \right]^\top.$ The state of the system is the concatenation of the agents' states
$$
x(t)= [x^1(t)^\top x^2(t)^\top \cdots x^n(t)^\top]^\top \in \Rset^{nm}.
$$
The stochastic process we analyze here is described by sequences of edges $\gamma=e_{i_1}\cdots e_{i_t} \cdots$ in $G$ with the convention that {if $e_{i_t}=(v_i,v_j)$, then agents $i$ and $j$ update their states according to a row stochastic matrix ${A_{ij}}$, called \textbf{local stochastic matrix}. 
Then, the gossip process on edge $e_{i_t}=(v_i,v_j)$ at time $t$ is given by } 
\begin{equation}\label{eqn:gossip_process}
x(t+1) = {A}_{ij} x(t).  
\end{equation}
The matrix $A_{ij}$ is an $nm$-dimensional stochastic matrix such that the rows and columns corresponding to the states of agent $i$ and agent $j$ is the submatrix $\tilde{A}_{ij}$ and the rows and columns corresponding to the other agents is the identity matrix. {We refer to the matrix $\tilde{A}_{ij}$ as~\textbf{pre-local stochastic matrix} for agents $i$ and $j$.} For example, the local stochastic matrix $A_{12}$, which is associated with the edge $(v_1,v_2)$, is given by
\begin{equation}\label{block}
A_{12} = 
\begin{bmatrix}
\tilde{A}_{12} & 0_{2m \times (n-2)m}\\
0_{(n-2)m \times 2m} & I_{(n-2)m \times (n-2)m}
\end{bmatrix}
\end{equation}
 We assume here that $A_{ij} = A_{ji}$. {If each edge $e$ in $G$ is labeled with some matrix $A_e$, then the graph $G$ is called a {\em matrix weighted graph}.
} Hence, when dealing with sequences of edges in $\vec G$, we associate $A_{ij}$ with both $v_iv_j$ and $v_jv_i$. This makes the graph $\Vec{G}$ a directed matrix weighted graph.

For a finite sequence $\gamma = e_1 \cdots e_k$ of edges in $G$ and for a given pair of integers
$0 \leq s \leq t \leq k$, we define the transition matrix $P_\gamma(t : s)$ for $t \ge s + 1$ as {the left product of local stochastic matrices from $s+1$ to $t$, given by}:
\begin{equation}\label{eqn:state_transition_over_walk}
P_{\gamma}(t : s) := A_{e_t} A_{e_{t-1}}\cdots A_{e_{s+1}}
\end{equation}
We set $P_\gamma(t : s) = I$ for $t \leq s$. This allows us to write the following update for the state vector $x$ at $s$:
\begin{equation}\label{eqn:state_transition}
x(t) = P_\gamma(t:s) x(s). 
\end{equation}
When clear from the context, we will simply write $P_\gamma$ for $P_{\gamma}(t : s)$.

{\textbf{Equivalence Classes of Cycles.}} A pointed cycle in $\vec G$ is a walk $v_{i_1}v_{i_2}\cdots v_{i_k}v_{i_1}$, where $v_{i_1}$ is called the {\em basepoint} of the cycle. Let $\vec{\mathcal{C}}_{\bullet}$ be the set of all pointed cycles in $\vec{G}$.

To each pointed cycle {$C \in \vec{\mathcal{C}}_{\bullet}$}, we assign a transition matrix $P_C$ as in~\eqref{eqn:state_transition_over_walk}; when we want to emphasize the basepoint, we write $P_{C,i}$ if the basepoint is $v_i$. 

Given a cycle $C$ in $\vec{G}$, we can reduce the dimension of vectors in $\Delta^{nm-1}$ and stochastic matrices $A \in \Rset^{nm \times nm}$ by removing rows and/or columns corresponding to nodes that are {\em not} covered by $C$. For example, if $C=v_1v_2v_3v_1$, then we let $\bar{A} \in \Rset^{3m \times 3m}$ be the principal submatrix of $A$ obtained by keeping the first $3m$ rows and columns; similarly, $\bar{w} \in \Rset^{3m}$ is the subvector of $w$ obtained by keeping the first $3m$ entries. The cycle $C$, and hence the dimension of the operation $\bar{\cdot}$, will always be clear from the context.

{\textbf{The Existence of Absorbing States.}} In this work, we allow a pre-local stochastic matrix to have zeros in any row. Consequently, it is possible to construct a valid local stochastic matrix by having a single non-zero element in any given row while setting all other elements in that row to zero (e.g. having a standard unit vector as a row). This leads to the possibility that a valid stochastic matrix can either be a complete permutation matrix or include a permutation block.

We can elaborate on the concept of a permutation block using the following definition{, where we denote the set of permutation matrices of size $n\times n$ by $S_n$}.  Given an index set $\pi \subseteq \{1,\ldots, nm\}$, a stochastic matrix $A$ has {\em a permutation block for $\pi$} if the submatrix of $A$ with rows/columns indexed by $\pi$ is a permutation matrix in $S_{|\pi|}$. Let ${\pi_A}$ be the largest index set among the sets indexing permutation submatrices in $A$; we refer to it as the {\em maximal permutation index of $A$}. 

To illustrate, consider the following instances of pre-local stochastic matrices, which are associated with the edges $e_1,e_2$ and $e_3$, respectively, in the graph $\Vec{G}$:
\begin{align}\label{eqn:example_matrices}
\tilde{A}_{e_1} &:= 
\left[\begin{smallmatrix}
0 & 1 & 0 & 0   \\
a_{21} &a_{22} &a_{23} &a_{24}  \\
1 & 0 & 0 & 0  \\
a_{41} &a_{42} &a_{43} &a_{44}
\end{smallmatrix}\right], 
\tilde{A}_{e_2} := 
\left[\begin{smallmatrix}
 0 & 0 & 1 & 0  \\
 1 & 0 & 0 & 0  \\
 b_{53} &b_{54} &b_{55} &b_{56} \\
 b_{63} &b_{64} &b_{65} &b_{66}
\end{smallmatrix}\right], \\
\tilde{A}_{e_3} &:= 
\left[\begin{smallmatrix} 
c_{11} & c_{12} & c_{15} & c_{16}  \\
1 & 0 &  0 & 0  \\
c_{51} & c_{52} & c_{55} & c_{56}  \\
c_{61} & c_{62}  & c_{65} & c_{66}  
\end{smallmatrix}\right],
\end{align}

with $a_{ij},b_{ij}$ and $c_{ij}$ real numbers in the open interval $(0, 1)$. Note that none of the pre-local matrices provided in~\eqref{eqn:example_matrices} has a permutation block for any index set; in contrast, they have rows that have a single nonzero entry. we will address these matrices in~\eqref{eqn:example_matrices} in later this section.

Within the state transition matrix definition~\eqref{eqn:state_transition_over_walk}, we can elaborate on the concept of permutation block. Consider a pointed cycle $C_1$ to be a walk such that $C_1:=e_3e_2e_1$. Assume that the associated pre-local stochastic matrices with the edges in $C_1$ are provided in~\eqref{eqn:example_matrices}. According to~\eqref{eqn:state_transition_over_walk}, we then have the following:
\begin{equation}
\bar{P}_{C,1} = \bar{A}_{e_1}\bar{A}_{e_2}\bar{A}_{e_3} = \left[\begin{smallmatrix}
1 & 0 & 0 & 0 & 0 & 0  \\
d_{21} & d_{22} & d_{23} &  0 & d_{25} & d_{26}  \\
d_{31} & d_{32} & d_{33} &  0 & d_{35} & d_{36}  \\
d_{41} & d_{42} & d_{43} &  0 & d_{45} & d_{46}  \\
d_{51} & d_{52} & d_{53} &  d_{54} & d_{55} & d_{56}  \\
d_{61} & d_{62} & d_{63} &  d_{64} & d_{65} & d_{66}  
\end{smallmatrix}\right]
\end{equation}
with $d_{ij}$ real numbers in the open interval $(0, 1)$. The matrix $\Bar{P}_{C,1}$ has a permutation block for the index set $\pi=\{1\}$. 

It is worth noting that, in contrast, none of the pre-local stochastic matrices provided in~\eqref{eqn:example_matrices}, which are associated with the edges in $C_1$, have a permutation block for any index set. This observation underlines that even though none of the pre-local stochastic matrices for the edges in $C_1$ conform to the criteria of a permutation block, the corresponding state transition matrix for the pointed cycle $C_1$ distinctly features such a permutation block. 


\section{Main Result}\label{sec:main_resutls}
 
In this section, we introduce the main concepts and present the main result of this paper. As already mentioned, an important concept is the one of {\em holonomy}. In differential geometry, holonomy deals with the variation of some quantity (e.g. a vector) along the loops in a given space. If there is no variation in this quantity after completing a loop, the process is defined as {\em holonomic}. Otherwise, it is said to be {\em non-holonomic}. In our convention, if there exists a finite $k>1$ such that the quantity does change after completing the loop once but comes back to the initial value after completing the loop $k$-times, the process which evolves the said quantity is called {\em finitely non-holonomic}. For our purpose, the quantity is a left weight vector, the process evolving the quantity is the gossip process, and the space is the graph $\vec G$. 

\textbf{Holonomy in the network.} {We first define an equivalence relation on the set of all pointed cycles, i.e. ${\vec{\mathcal{C}}_{\bullet}}$ by saying that $C_i, C_j \in \vec{\mathcal{C}}_{\bullet}$ are equivalent if they visit the same vertices in the same cyclic order. The set of equivalence classes of pointed cycles are referred to as {\em cycles}, and  we denote by $\vec{\mathcal{C}}$ the set of all such equivalence classes.} {For example, the pointed cycles $C_i = v_i v_1 \cdots v_2 v_j v_i$ and $C_j = v_j v_i v_1 \cdots v_2 v_j$ belong to the same equivalence class, since they differ only in the choice of starting vertex, and therefore represent the same cycle.}

{
Note that for a given cycle $C \in \vec{\mathcal{C}}$, the state transition matrix  $P_{C,i}$ depends on the choice of basepoint of the pointed cycle $C_i \in C$. In other words, different basepoints yield cyclic permutations of the same matrix product, which are in general not equal. Hence, by definition, the existence of a left fixed point, i.e.
$w = w(P_{C,i})^k$, is a property of the pointed cycle $C_i$, not of the cycle $C$ itself. However, the following lemma shows that this property is {\em in fact} shared by all pointed cycles in the same equivalence class. Therefore, this property is well defined for the cycles.} 


\begin{lem} \label{lem:path_change}
Let $C_i$ and $C_j$ be two pointed cycles in a cycle $C$. If there exists a weight vector $w$ such that $w=w (P_{C,i})^k$ holds for some positive $k$, then there exists a weight vector $w'$ such that $w'=w' (P_{C,j})^k$ holds. 
\end{lem}
\begin{proof} {Without loss of generality, assume that $v_i$ and $v_j$ are adjacent; the general 
case follows by iterating the argument along the cycle.} Let $C=v_iv_1\cdots v_jv_i$ be a cycle in $\vec{G}$. Consider the pointed cycles $C_i = v_i v_1 \cdots v_2 v_j v_i$ and $C_j = v_j v_i v_1 \cdots v_2 v_j$ in the cycle $C$. According to the statement, it holds that:
\begin{equation}\label{eqn:closed walk_i}
  w = w ({P}_{C,i})^k  \mbox{ where }  
{P}_{C,i} =  {A}_{v_jv_i} {A}_{v_2v_j} \cdots {A}_{v_iv_1}. 
\end{equation}
By multiplying~\eqref{eqn:closed walk_i} by the matrix ${A}_{v_jv_i}$ from the right, we get
\begin{align*}
     w {A}_{v_jv_i} &= w ( {A}_{v_jv_i} {A}_{v_2v_j} \cdots {A}_{v_iv_1})^k  {A}_{v_jv_i}  \\
     &= w  {A}_{v_jv_i} ( {A}_{v_2v_j} \cdots {A}_{v_iv_1} {A}_{v_jv_i} )^k = w {A}_{v_jv_i}  ({P}_{C,j})^k.  
\end{align*}
It is clear that the product $w {A}_{v_jv_i}$ is a weight vector. This shows that there exists a weight vector $w'$ such that $w'= w' ({P}_{C,j})^k$ and it is equal to $w  {A}_{v_jv_i}$.
\end{proof}

Paraphrasing the statement says that the condition $w = w(P_{C,i})^k$ is an invariant of the equivalence class $C \in \vec{\mathcal{C}}$. To be more precise, if it holds for one pointed cycle $C_i \in C$, then it holds (with a possibly different weight vector) for every pointed cycle $C_j \in C$. { With this fact, we introduce the notion of holonomy for a cycle as follows:}




\begin{defn}[Holonomic Stochastic Matrices]\label{def:non-trivial} Let $C$ be a cycle in $\vec{G}$ of length greater than $2$ and $w^\top \in \operatorname{int}\Delta^{nm-1}$ be a weight vector. The {\em $w$-order} of $C$ is defined as 
$$
\operatorname{ord}_wC := \min \{k\geq 1: \bar{w}=\bar{w}(\bar P_C)^k\},
$$
and $\operatorname{ord}_w C=0$ if the set is empty. The local stochastic matrices $A_{e}$,  $e \in C$,  are said to be \textbf{$w$-holonomic for $C$} if there exists a weight vector $w$ such that $\operatorname{ord}_wC$ is finite and non-zero.
\end{defn}

We observe that if $\bar{w} = \bar{w} (\bar{P}_{C})^k$ holds for some positive integer $k$, then $\bar{w}=\bar{w} (\bar{P}_{C})^{nk}$ holds for all $n \in \Nset$, thus making the set of integers $k$ for which $\bar{w} = \bar{w} (\bar{P}_{C})^k$ holds of infinite cardinality.



Moreover, by Lemma~\ref{lem:path_change}, $w$-order of a cycle is independent from the choice of the basepoint. We denote \textbf{the orbit set of a weight vector $w$ around a cycle $C$} as $\mathcal{O}_w^C$, that is,
\begin{equation}\label{eqn:orbit_set}
    \mathcal{O}_w^C:=\{ w_C^{(a)} \in \mathbb{R}^{nm} | w_C^{(a)} = w({P}_C)^a \mbox{ for }  a \in \mathbb{N}  \}.
\end{equation}
{Note that if $a=0$, the weight vector $w^{(0)}_C$ is equal to $w$ by construction for any cycle $C$.}





{ We now extend the notion of holonomy from individual cycles to the full graph $\vec{G}$. By Lemma~\ref{lem:path_change}, the following conditions depend 
only on the equivalence class $C \in \vec{\mathcal{C}}$, and not on the choice of representative pointed cycle.}

The local stochastic matrices $A_e$ are {\em $w$-holonomic for $G$} if there exists a {\em common} weight vector $w$ such that
\begin{enumerate}
    \item For every cycle $C \in \vec{\mathcal{C}}$ of length at least $3$, the matrices $A_e$ are $w$-holonomic for $C$ (i.e., $|\mathcal{O}_w^C|$ is finite).
    \item For any two distinct cycles $C_1, C_2 \in \vec{\mathcal{C}}$ of length at least $3$, the orbit sets of $w$ satisfy
    $$
    \mathcal{O}_w^{C_1} \cap \mathcal{O}_w^{C_2} = \{w\}.
    $$
\end{enumerate}
The second condition ensures that the orbit sets of distinct cycles intersect trivially.

{From Definition~\ref{def:non-trivial}, the $w$-order of a given cycle $C$ can vary as a function of $w$. For our purpose, we need to consider the $w$'s that yield the largest $w$-order and thus define the {\bf order of a cycle} as
\begin{equation}\label{eqn:order_def}
    \operatorname{ord} C := \sup_{w^\top \in \operatorname{int}\Delta^{nm-1}} \operatorname{ord}_w C. 
\end{equation}
Now, we can define the holonomy of a gossip process as follows.
\begin{itemize}
    \item If $\operatorname{ord} C = 0$, then the process on $C$ is {\em non-holonomic}.
    \item If $\operatorname{ord} C = 1$, then the process on $C$ is {\em holonomic}.    
    \item If $\operatorname{ord} C > 1$, then the process on $C$ is {\em finitely non-holonomic}.  
\end{itemize}
One can assign a (holonomy) group to the process on the cycle $C$ if $\operatorname{ord} C \geq 1$. If $\operatorname{ord} C = 1$, the cycle $C$ is said to have \textbf{trivial holonomy} since the corresponding group has only identity operation (trivial group). If $\operatorname{ord} C > 1$, the cycle $C$ is said to have \textbf{non-trivial holonomy} since the corresponding group is a cyclic group with an order greater than one.}

\textbf{Graph topology.} In an undirected graph $G$, two nodes $v_i$ and $v_j$ are called connected if the graph $G$ contains a path from $v_i$ to $v_j$. We need the following notion.

\begin{defn}[Bridge]\label{def:bridge}
Let $G=(V,E)$ be an undirected graph. Let $\mathcal{S}_G$ be the set of pairs of nodes that are connected {by a path} in $G$. Let $\Tilde{G}_e$ be the undirected graph obtained by removing the edge $e$ from $G$, that is, $\Tilde{G}_e=(V,E\setminus \{e\})$. If $|\mathcal{S}_{G}|$ is strictly greater than $|\mathcal{S}_{\Tilde{G}_e}|$, the edge $e$ is called a \textbf{bridge} (cut-edge) of $G$.
\end{defn}
A graph $G$ without a bridge is called \textbf{bridgeless}. {Connected bridgeless graphs are also equivalent to 2-edge-connected graphs~\cite[Thm 3.15]{merris2011graph}.} We record the following result characterizing cut-edges.
\begin{prop}\label{prop:bridge}
An edge $e$ in a connected graph $G$ is a bridge if and only if no cycles of $G$ contain both vertices adjacent to $e$.
\end{prop}

See~\cite[Theorem 3.3]{merris2011graph} for a proof of Proposition~\ref{prop:bridge}. Paraphrasing, the statement says that every node in a connected, simple, bridgeless graph $G$ is covered by at least one cycle.

\textbf{Derived graphs and $\psi(\cdot)$ map.} We now focus on describing the allowable sequences of updates which yields the consensus at the limit for the gossip process. These will be defined via paths in what we call the {\em derived graph} of $G$ by $w$, denoted by $D_G(w)$. We present this graph as a {\em geometric} graph, with nodes embedded in $\mathbf{R}^{nm}$.

\begin{defn}[Derived Graph $D_G(\cdot)$]\label{def:derived_graph_w}
Let $G=(V,E)$ be a {\em matrix weighted} graph on $n$ nodes, with weights $A_e \in \mathbb{R}^{nm \times nm}$.
For a weight vector $w^\top \in \operatorname{int}\Delta^{nm-1}$, the derived graph of $G$ generated by $w$, denoted by $D_G(w)= (N_w, \vec{E}_w)$, is a directed matrix weighted graph, possibly with multi-edges and self-loops, with $N_w = \bigcup_{C \in \vec{\mathcal{C}} } \mathcal{O}_w^C$. 
{For $w_C^{(i)}, w_C^{(j)} \in \mathcal{O}_w^C$, there exists an edge $w_C^{(i)}w_C^{(j)} \in \vec{E}_w$ if $w_C^{(i)}=w_C^{(j)} P_C$ for a (pointed) cycle $C$; in this case, the edge weight is $P_C$. }
\end{defn}
{
The first condition in the definition of $w$-holonomy for $G$ (i.e. $|\mathcal{O}^C_w|$ is finite for all $C \in \vec{\mathcal{C}}$) ensures that the derived graph has a finite set of nodes.}

\begin{rem}\label{rem:geometric_graph}
Note that the derived graph being a geometric graph, ensures that elements of distinct orbit sets with the same coordinates {(i.e. the weight vector $w$)} correspond to a unique vertex in the derived graph. 
\end{rem}

{We introduce the mapping $\psi(\cdot)$, which assigns to an edge in $D_G(w)$ a cycle in $G$. Precisely,  let $e \in \vec{E}_w$ have weight $P_C$ with  $C=v_iv_{i+1}\cdots v_kv_i$. We set $$\psi(e)= v_iv_{i+1}\cdots v_kv_i.$$  We extend the domain of $\psi$ to the set of paths in ${D}_G(w)$ according to 
$$ \psi(\gamma \lor e) = \psi(e) \lor \psi(\gamma)$$
for any walk $\gamma$ in ${D}_G(w)$. } 
By a possible abuse of notation, we represent the path $\psi(e)$ by its vertex sequence while concatenation operation $\vee$ remains defined on the underlying edges.

Note that the order reversal in the above equation, {is} a change that is essential for maintaining coherence. The gossip process evolves as the left multiplication of local stochastic matrices while the paths in the derived graph $D_G(w)$ correspond to the right multiplication of matrices with row vectors.

We provide an example for the derivation of $D_G(w)$. 
\begin{exmp}\label{exmp:derivation}
Consider a simple, connected, bridgeless graph with matrix weights $A_e \in \mathbb{R}^{nm\times nm}$ as depicted in Figure~\ref{fig:butterfly_graph}. Consider the pointed cycles $C_1=v_1v_2v_3v_1$, $C_2=v_4v_1v_5v_4$ and $C_3=v_5v_7v_6v_5$ in $G$. Assume that the set of local stochastic matrices $A_e, e \in E$ is $w$-holonomic for $G$ and the corresponding orbit sets of the weight vector $w$ around each of these cycles are  
        \begin{align*}
          \mathcal{O}_w^{C_1} &= \{ w , w^{(1)}_{C_1} , \cdots ,  w^{(k_1-2)}_{C_1} ,w^{(k_1-1)}_{C_1} \} \\
         \mathcal{O}_w^{C_2} &= \{ w , w^{(1)}_{C_2} , \cdots , w^{(k_2-2)}_{C_2}  , w^{(k_2-1)}_{C_2} \} \\
          \mathcal{O}_w^{C_3} &= \{ w \}
        \end{align*}
where the $w$-order of the cycles $C_1$, $C_2$ and $C_3$ are $k_1,k_2$ and $1$, respectively.
\begin{figure}
\centering
\begin{tikzpicture}[scale=0.9]
		\node [circle,fill=black,inner sep=1pt,label=above:{\footnotesize $v_1$}] (b_i) at (3, 0) {};
		
		\node [circle,fill=black,inner sep=1pt,label=above:{\footnotesize $v_2$}] (b_j) at (1, 0.5) {};

		\node [circle,fill=black,inner sep=1pt,label=below:{\footnotesize $v_3$}] (b_l) at (1, -0.5) {};
        \node [circle,fill=black,inner sep=1pt,label=above:{\footnotesize $v_4$}] (b_g) at (5, 0.5) {};
        \node [circle,fill=black,inner sep=1pt,label=below:{\footnotesize $v_5$}] (b_q) at (5, -0.5) {};
	    
        \node [circle,fill=black,inner sep=1pt,label=below:{\footnotesize $v_6$}] (b_y) at (7, -0.5) {};
        \node [circle,fill=black,inner sep=1pt,label=above:{\footnotesize $v_7$}] (b_u) at (7, 0.5) {};

   \path[draw,,shorten >=2pt,shorten <=2pt]
		(b_i) edge[<->] (b_j)

        (b_j) edge[ <->] (b_l)
        (b_l) edge[<->] (b_i)	

        (b_i) edge[<->] (b_g)
        (b_g) edge[<->] (b_q)
        (b_i) edge[<->] (b_q)

        (b_y) edge[<->] (b_q)
        (b_y) edge[<->] (b_u)
        (b_u) edge[<->] (b_q)

		 ;
\end{tikzpicture}
\caption{The graph $\vec{G}$ } 
\label{fig:butterfly_graph}
\end{figure}

\begin{figure}
\centering
\begin{tikzpicture}[scale=0.9]
		
		\node [circle,fill=black,inner sep=1pt,label=above:{\footnotesize  ${w}_{C_2}^{(1)}$ }] (b_1) at (6, 1) {};
		\node [circle,fill=black,inner sep=1pt,label=below:{\footnotesize  ${w}_{C_2}^{(k_2-1)}$ }] (b_j) at (6, -1) {};

		\node [circle,fill=black,inner sep=1pt,label=above:{\footnotesize  ${w}_{C_2}^{(k_2-2)}$ }] (b_2) at (8, 0) {};

		\node [circle,fill=black,inner sep=1pt,label=below:{\footnotesize ${w}$}] (b_a) at (4, 0) {};


		\node [circle,fill=black,inner sep=1pt,label=above:{\footnotesize  ${w}_{C_1}^{(1)}$ }] (b_l) at (2, 1) {};
		\node [circle,fill=black,inner sep=1pt,label=below:{\footnotesize  ${w}_{C_1}^{(k_1-1)}$ }] (b_m) at (2, -1) {};

		\node [circle,fill=black,inner sep=1pt,label=above:{\footnotesize  ${w}_{C_1}^{(k_1-2)}$ }] (b_n) at (0, 0) {};
	
   \path[draw,every node/.style={sloped,anchor=south,auto=false},shorten >=2pt,shorten <=2pt]
		(b_a) edge[->] node {$e_{C_2}^{0}$} (b_1)
        (b_1) edge[dashed, ->]  (b_2)
		 (b_2) edge[->] node {$e_{C_2}^{k_2-2}$} (b_j)		 
		 (b_a) edge[<-] node {$e_{C_2}^{k_2-1}$} (b_j)

   (b_a) edge[->] node {$e_{C_1}^{0}$} (b_l)

		(b_a) edge[loop above, ->] node {$e_{C_3}$} (b_a)

      (b_n) edge[->] node {$e_{C_1}^{k_1-2}$} (b_m)
  (b_l) edge[dashed, ->]  (b_n)   
  (b_m) edge[->] node {$e_{C_1}^{k_1-1}$} (b_a)		 ;
\end{tikzpicture}
\caption{The graph $D_G(w)$} 
\label{fig:derived_graph_w}

\end{figure} 
We denote an edge $e$ in the graph $D_G(w)$ as $e_{C_i}^{k}$ if its weight is the matrix $P_{C_i}$ and its starting node is $w_{C_i}^{(k)}$. By abuse of notation, we denote a self-loop $e$ in the graph $D_G(w)$ as $e_{C_i}$ if its weight is the matrix $P_{C_i}$. Consider the path $\gamma:=  e_{C_1}^{k_1-1} \lor e_{C_3} \lor  e_{C_2}^{0}$ in the graph $D_G(w)$. ``Travelling'' over the path $\gamma$ in $D_G(w)$ translates into the matrix product ${w}_{C_1}^{(k_1-1)} P_{C_1} P_{C_3} P_{C_2} (={w}_{C_2}^{(1)})$. {Now, recall the definition of $\psi(\cdot)$ to show how to map a sequence of edges in derived graph $D_G(w)$ to a sequence of edges in $G$. We have the following:
\begin{align}
    \psi(\gamma) &=  \psi(e_{C_1}^{k_1-1} \lor e_{C_3} \lor  e_{C_2}^{0} ) \\  &=  \psi( e_{C_3} \lor  e_{C_2}^{0} ) \lor \psi(e_{C_1}^{k_1-1} ) \\ &=  \psi( e_{C_2}^{0} ) \lor  \psi( e_{C_3} ) \lor   \psi( e_{C_1}^{k_1-1} ) =  C_2C_3C_1
\end{align}} 
In this example, we showed how to derive $D_G(w)$ and map a given walk $\gamma$ in $D_G(w)$ to a sequence of edges in $G$ via $\psi(\cdot)$.  \end{exmp}

{We need the notion of an exhaustive walk to state the main result. We have adapted the following definition from~\cite{belabbas2021triangulated}:}

{\begin{defn}[Exhaustive Walk]\label{def:exhaustive}
A finite walk $\gamma$ in $D_G(w)$ is exhaustive if it visits every node of $D_G(w)$ at least once. An infinite walk $\gamma$ in $D_G(w)$ is exhaustive if it visits every node of $D_G(w)$ infinitely often.
\end{defn}}

\textbf{Main Theorem:} We can now present our main theorem which provides the conditions for the convergence of a non-homogeneous Markov process in the presence of absorbing states and non-trivial holonomy. { In particular, it provides the accumulation points of the infinite products $P_{\psi(\gamma)}$.}

\begin{thm}\label{thm:consensus}
Let $G=(V,E)$ be a simple, connected, bridgeless graph on $n$ nodes with matrix-valued edge weights $A_e$, $e \in E$.  Let  $w^\top \in \operatorname{int}\Delta^{nm-1}$ be such that the set of local stochastic matrices $\{A_{e}\in \mathbb{R}^{nm\times nm}, e \in E\}$ is $w$-holonomic for $G$. Then, for any infinite exhaustive walk $\gamma$ in $D_G(w)$, we have the following:

\begin{enumerate}
\item\label{itm:limit_set} The limit set of ${P}_{\psi(\gamma)}$ is a finite set $\mathcal{L}$. 

\item\label{itm:block} There exists a relabeling of the nodes such that each element of $\mathcal{L}$ can be expressed as   
\begin{equation}\label{eqn:limit_set}
    P_{\psi(\gamma)} = \left[\begin{array}{cc}
\tilde{P}_{\psi(\gamma)} & 0 \\
0 & M_{\psi(\gamma)} 
\end{array}\right] 
\end{equation}
where $\tilde{P}_{\psi(\gamma)} $ is a permutation matrix with rows columns indexed by the set $\cap_{C \in \vec{\mathcal{C}}}\pi_{P_C}$ 
and $M_{\psi(\gamma)}$ is a block diagonal matrix.   
\item\label{itm:rank_one} The blocks $M_{\psi(\gamma)}^{ii}$ of $M_{\psi(\gamma)}$ are rank-one matrices. 
\end{enumerate}
\end{thm}
{

Theorem~\ref{thm:consensus} states that, under the conditions given, the infinite product of stochastic matrices (possibly having absorbing states) whose order is provided by $\psi(\gamma)$ converges to a matrix in the limit set $\mathcal{L}$ for any infinite exhaustive walk $\gamma$ in the derived $D_G(w)$ and this limit set has finite cardinality (See assertion~\ref{itm:limit_set} in Theorem~\ref{thm:consensus}). In other words, any infinite exhaustive walk $\gamma$ in the derived $D_G(w)$ provides {\em an allowable sequence of edges} in $G$ via the mapping $\psi(\cdot)$, which corresponds to an allowable sequence of the updates for the gossip process by the associated matrix weights with edges in $G$.

Additionally, in Lemma~\ref{lem:exh_walk}, we show that any random walk in the derived graph is an exhaustive walk with probability one. Then, we have the following corollary to Theorem~\ref{thm:consensus} and Lemma~\ref{lem:exh_walk}:
\begin{cor}\label{cor:random_walk}
For a random walk $\gamma$ in $D_G(w)$, the limit set $\mathcal{L}$ of ${P}_{\psi(\gamma)}$ is finite  with probability one.  
\end{cor}
 
This corollary ensures that a sequence of updates provided by a random walk in the derived graph yields consensus in the limit for the gossip process with probability one. In other words, the gossip process maintains the same consensus properties even when the updates are generated by a random walk in the derived graph. 

Furthermore, Theorem~\ref{thm:consensus} describes the limit set $\mathcal{L}$ as follows: any matrix in the limit set $\mathcal{L}$ can be represented as a block diagonal matrix, with each block being either a permutation matrix or a rank-one matrix, up to some relabeling (See assertions~\ref{itm:block} and~\ref{itm:rank_one} in Theorem~\ref{thm:consensus}). }

{
\begin{rem}\label{rem:explicilt}
We provide an explicit description of the limit set $\mathcal{L}$ and the block $M_{\psi(\gamma)}^{ii}$ as a function of $w$ in the proof of Theorem~\ref{thm:consensus} in the next section. We show that there exists a partition of indexed set induced by $\psi(\gamma)$, denoted by $\pi^{\psi(\gamma)}$, such that 
$M_{\psi(\gamma)}^{ii} = \mathds{1} p_i^\top $ where
\begin{align*}\\[-2.2em]
    p_i^\top := \frac{ [{w}]_{j \in \pi^{\psi(\gamma)}_i}   }{\alpha_i} \mbox{ where } \alpha_i:=\sum_{j \in \pi^{\psi(\gamma)}_i} {w}_j.\\[-2.2em]
\end{align*}
The partition of the index set $\pi^{\psi(\gamma)}$ can be computed following the method described in~Section~\ref{sec:proof}. 
\end{rem}
}

{Theorem~\ref{thm:consensus}, particularly assertion~\ref{itm:rank_one}, provides a class-ergodic consensus (see~\cite[Defn. 4]{bolouki2013ergodicity}), in which the partition of the index set $\pi^{\psi(\gamma)}$ corresponds to the consensus cluster of the states of agents as discussed in~\cite{bolouki2013ergodicity,touri2012approximations}.}

The descriptions of the blocks $M_{\psi(\gamma)}^{ii}$ guarantee that they have no zero entries. It follows that if the transition matrix $P_C$ for a cycle $C$ has a $1$ in a row (e.g. having a standard unit vector as a row), then the $1$ in the row either shows up in the permutation block $\tilde{P}_{\psi(\gamma)}$ for a matrix in the limit set or disappears by converging a value as a function of $w$ in the block $M_{\psi(\gamma)}$. 

Furthermore, in our recent work~\cite{bayram2025construct}, we have introduced a numerical algorithm to construct (local) stochastic matrices that are $w$-holonomic for a given gossip network topology and a given weight vector for averaging across different consensus clusters.


\section{Analysis and Proofs of Theorems}\label{sec:proof}
In this section, we analyze the cycles that have nonzero order and then prove Theorem~\ref{thm:consensus}. The major contribution of this section is to develop a necessary condition for a cycle to have nonzero order and to provide a proof of Theorem~\ref{thm:consensus}. We start with an analysis of the spectrum of the product of local stochastic matrices and its relationship with holonomy.

\subsection{Cycles with Nonzero Order}\label{subsec:nonzero_order}
Let  $S^1 := \{z \in \Cval: |z| = 1 \}$ and let $S^{1^*} := S^1 \backslash  \{1\}$. The following lemma provides a necessary condition for a cycle to have non-trivial holonomy in $G$.
\begin{lem} \label{lem:unit_circle}
If a cycle $C$ has non-trivial holonomy, then the matrix $\bar{P}_{C}$ has at least one eigenvalue in the set $S^{1^*}$.
\end{lem}
\begin{proof}
Let $\sigma(\bar{P}_{C})$ be the spectrum of the matrix $\bar{P}_{C}$. If $C$ has non-trivial holonomy, then there exists  $\bar{w}$ such that $\bar{w}=\bar{w}(\bar{P}_{C})^k$, for some $k>1$  and $\bar{w}\neq \bar{w}(\bar{P}_{C})$. The former condition implies that $\exists \lambda \in \sigma(\bar{P}_{C})$ such that $\lambda^k = 1$ or, equivalently, $\lambda \in S^1$ and the latter implies that $\lambda \neq 1$, from which the result follows. 
\end{proof} 
Motivated by the previous Lemma, we now seek to describe stochastic matrices whose spectra have non-empty intersections with $S^{1^*}$. Note that the spectral radius of a stochastic matrix is $1$ and its spectral circle is $S^1$. 

{A matrix $A$ with order $n$ is {\em reducible} if there exists a permutation matrix $P \in S_n$ such that
\begin{equation}\label{eqn:factorization}
   P^\top A P =\left[\begin{array}{cc}
A_{11} & A_{12} \\
0 & A_{22}
\end{array}\right] 
\end{equation}
where $A_{11}$ and {$A_{22}$} are nonempty square matrices. If $A$ is not reducible, then $A$ is called {\em irreducible}. For convenience, we denote similarity through the permutation matrix $P$ by $\sim_P$.}

Let $A$ be a reducible matrix. We know that $A \sim_P B$ where $B$ is {a block upper triangular matrix} of the form:
\begin{equation}\label{eqn:canonical_form}
  B=\left[ \begin{smallmatrix}
B_{11}  & \cdots &  B_{1r} &  B_{1,r+1} & \cdots   & B_{1m}  \\
 & \ddots &  \vdots  & \vdots & \vdots &  \vdots \\
 &  &  B_{rr}   & B_{r,r+1} & \cdots   &  B_{rm} \\
 &   &   &  B_{r+1,r+1} &  &  \textbf{0}  \\
& & &  & \ddots &  \\
 \textbf{0}  &   &    &    &  & B_{mm} \\
\end{smallmatrix}\right] 
\end{equation}
where each principal block $B_{ii}$ is either irreducible or the zero matrix~\cite{meyer2000matrix};  we say that the form of $B$ given in~\eqref{eqn:canonical_form} is a {\em canonical form for reducible matrices}. In the literature, the subset of states corresponding to $B_{kk}$ for $1 \leq k \leq r$ is called the $k^{th}$ {\em transient class} of matrix $B$ and the subset of states corresponding to $B_{r+j,r+j}$ for $j \geq 1$ is called the $j^{th}$ {\em ergodic class} of matrix $B$. 
\begin{lem}\label{lem:irreducible_blocks}
Let $C$ be a cycle with nonzero order. Then, there exists a permutation matrix $P$ such that $\bar{P}_C \sim_P B$ where $B$ is as in~\eqref{eqn:canonical_form} with $r=0$ and $m \geq 1$.
\end{lem}
\begin{proof} 
If the matrix $\bar{P}_C$ is irreducible, then the result trivially holds for $P=I$. If the matrix $\bar{P}_C$ is reducible, by definition, there exists a permutation matrix $P$ such that $\bar{P}_C \sim_P B$ where $B$ as in 
\eqref{eqn:canonical_form}. The principal submatrices $B_{ii}$ for $i \leq r$ act solely on the transient {class} of $\bar{P}_C$. Therefore, $\rho(B_{ii}) < 1$ for $i \leq r$. This then implies that the matrices $B_{ii}$, $i \leq r$ are convergent. Since $\operatorname{ord} C > 0$, the matrix $B$ cannot have convergent submatrices on the diagonal, from which the result follows.
\end{proof}
Paraphrasing, the statement says that the matrix $\bar{P}_C$ is permutationally similar to a block-diagonal matrix $B$, where the principal blocks are irreducible matrices. {Moreover, it contains no convergent matrix blocks, since $r$, which is the index of the last convergent block in~\eqref{eqn:canonical_form}, is zero and $m$, which is the index of the last irreducible block, is greater than zero. Then, all the principal diagonal blocks belong to the ergodic class of matrix $B$, as a result of which the off-diagonal blocks are zero, so the principal blocks themselves form a stochastic matrix.} Note that a permutation matrix is an irreducible stochastic matrix. Isolating the permutation part of $P_C$, we can further write, up to relabeling, that 
\begin{equation}\label{eqn:factorization_Pc}
  P_C \sim_P \!\left[\begin{smallmatrix}
B^{00}_C & 0\\
0 & M_C
\end{smallmatrix}\right] \!=\!\begin{blockarray}{ccccc}
\pi^C_0     & \pi^C_1 & \cdots  & \pi^C_m    &  \\
\begin{block}{[cccc]c}
B^{00}_C    &          &        & \textbf{0} & \pi^C_0 \\
            & B^{11}_C &        &            & \pi^C_1 \\
            &          & \ddots &            & \vdots \\
\textbf{0}  &          &        & B^{mm}_C   & \pi^C_m  \\
\end{block}
\end{blockarray}%
\end{equation}
where the principal submatrix $B^{00}_C$ is a permutation matrix. For convenience, we denote the irreducible blocks $B_{r+j,r+j}$ in \eqref{eqn:canonical_form} by $B^{jj}_C$.

{Let $\pi^C_j$ be the set of indices labeling the rows/columns in matrix $B$ that forms the block matrix ${B^{jj}_C}$ after relabelling $P_C$ through the permutation matrix $P$.} It then follows that we have a partition of the index set $\{1,\cdots,nm\}$ induced by $C$, denoted by $\pi^C$, precisely $\pi^C:=\{\{\pi^C_i\}_{i=0}^m\}$. {As a matter of convention, we use the notation $\pi^C_0$ interchangeably with $\pi_{P_C}$ to refer to the largest index set among the sets that index the permutation submatrices in $P_C$. We will revisit the concept of the partition of the index set induced by a cycle in the next subsection.} We now aim to better understand irreducible principal blocks of the matrix $M_C$~{in~\eqref{eqn:factorization_Pc}}. Hence, without loss of generality, we can assume that $\bar{P}_C$ is irreducible (e.g., only the $B^{11}_C$ block is nontrivial).

More is known about irreducible stochastic matrices. Let $A$ be an irreducible stochastic matrix; then there exists a unique vector $p$ satisfying
\begin{equation}\label{eqn:perron}
    A p = p \mbox{,  } p>0 \mbox{ and } \|p\|_1=1 
\end{equation}
which is called the {\em Perron vector}. {A stochastic matrix $A$ is called a scrambling matrix if no pair of rows of $A$ are orthogonal~\cite{chen2022gossip}.} The Perron-Frobenius Theorem says that an irreducible {\em primitive} stochastic matrix $A$ converges to a scrambling matrix whose rows are equal to the transpose of the Perron vector of $A^\top$~\cite{meyer2000matrix}, which is called the row Perron vector of the matrix $A$. To be more precise, $\lim_{k \to \infty} A^k = \mathds{1} q^\top$ where $A^\top q=q$.

{
The spectral radius of a matrix $A$ is the maximum of the modulus of the elements of its spectrum, denoted by $\rho(A)$. A circle on $\mathbb{C}$ with radius $\rho(A)$ is called {\em spectral circle} of the matrix $A$. A nonnegative irreducible matrix $A$ having $h>1$ eigenvalues on its spectral circle is called {\em imprimitive}, and $h$ is referred to as the index of imprimitivity. If there is only one eigenvalue on the spectral circle of $A$, then the matrix $A$ is {\em primitive}.
}

We now recall an extension of the Perron-Frobenius Theorem.
\begin{lem}[Frobenius Form]\label{lem:fro_form}
For each imprimitive matrix $A$ with index of imprimivity $h > 1$, there exists a permutation matrix $P$ such that $A \sim_P F$ where $F$ is of the form:
\begin{equation}
F=\left[\begin{smallmatrix}
\textbf{0} & A_{12} & 0 &  \cdots& 0 \\
0 & \textbf{0} & A_{23} &  \cdots  & 0 \\
\vdots &  \vdots & \ddots & \ddots &\vdots \\
0 &  0 & \cdots & \textbf{0} & A_{h-1,h} \\
A_{h1} &  0 & \cdots & 0 & \textbf{0}
\end{smallmatrix}\right],   
\end{equation}
where the zero blocks, denoted by $\textbf{0}$, on the main diagonal are square.     
\end{lem}
See~\cite{meyer2000matrix} for a proof of Lemma~\ref{lem:fro_form}. This is known as {\em Frobenius Form} for an irreducible imprimitive matrix. Then, the matrix $F^h$ is a block diagonal matrix with blocks that are primitive.

\subsection{Permutation Blocks}\label{subsec:perm_block}
We now aim to gain a better understanding of permutation matrices. We demonstrate that the block matrix $B^{00}_C$ is similar to a matrix that exhibits a permutation block structure after relabeling in~\eqref{eqn:factorization_Pc}. The following proposition establishes that the matrix $B^{00}_C$ itself possesses a permutation block structure. 

{We call a matrix $A$ conjugate to a permutation matrix if there exists a matrix $D$ in the group of invertible matrices, denoted by $\mathbb{GL}(n)$, and a permutation matrix $P\in S_n$ such that $A=DPD^{-1}$}

\begin{prop}\label{prop:no_similiar_to_perm}
If $A$ is a stochastic matrix which is conjugate to a permutation matrix, then $A$ is a permutation matrix.
\end{prop}
To prove {Proposition}~\ref{prop:no_similiar_to_perm}, we use the following lemma.
\begin{lem} \label{lem:invertible}
A matrix $A$ is a stochastic matrix with $A^{-1}$ also a stochastic matrix if and only if $A$ is a permutation matrix.
\end{lem}
See~\cite{teaching_tip} for a proof of Lemma~\ref{lem:invertible}. As an easy corollary, we have the following lemma:
\begin{lem} \label{lem:perm_matrix}
A matrix $A$ is a stochastic matrix with all its eigenvalues on the unit circle if and only if $A$ is a permutation matrix. 
\end{lem}
Using Lemma~\ref{lem:invertible}, we now prove Proposition~\ref{prop:no_similiar_to_perm}:
\begin{proof}[Proof of Proposition~\ref{prop:no_similiar_to_perm}.] Let $A=DPD^{-1}$ where $P \in S_n$, and $D \in \mathbb{GL}(n)$. We have that
$
A^k = D P^k D^{-1}.
$
Since $P$ is a permutation matrix, there exists a $k\geq1$ so that $P^k = I$. Putting the previous two equalities together, we get
$$
A^k = D I D^{-1} = D D^{-1} = I. 
$$
Hence, we have $A^{-1}=A^{k-1}$. Furthermore, by Lemma~\ref{lem:invertible}, the matrices $A^{-1}$ and $A$ are permutation matrices.     
\end{proof}
 { In this subsection, by Proposition~\ref{prop:no_similiar_to_perm}, we establish that a matrix conjugate to a permutation matrix is itself a permutation matrix. Thus, the principal block submatrix $B^{00}_C$ of $P_C$ not only exhibits a conjugacy to a permutation block but also directly represents a permutation structure derived from $B$.}

\subsection{Non-Permutation Blocks}

\begin{prop}\label{prop:block_irreducible_for_exhaustive}  
For an exhaustive walk $\gamma$ in the derived graph ${D}_G(w)$, there exists a permutation matrix $P$ such that 
\begin{equation}\label{lem:block_for_exhaustive}
P_{\psi(\gamma)} \sim_P \begin{blockarray}{ccccc}
\pi^{\psi(\gamma)}_0     & \pi^{\psi(\gamma)}_1 & \cdots  & \pi^{\psi(\gamma)}_m    &  \\
\begin{block}{[cccc]c}
B^{00}_{\psi(\gamma)}    &          &        & \textbf{0} & \pi^{\psi(\gamma)}_0 \\
            & B^{11}_{\psi(\gamma)} &        &            & \pi^{\psi(\gamma)}_1 \\
            &          & \ddots &            & \vdots \\
\textbf{0}  &          &        & B^{mm}_{\psi(\gamma)}   & \pi^{\psi(\gamma)}_m  \\
\end{block}
\end{blockarray}%
\end{equation} where $B_{\psi(\gamma)}^{ii}$ are irreducible matrices for $i \geq 1$, the rows/columns of which are indexed by the set $\pi^{\psi(\gamma)}_i$ for $i \geq 1$ and the rows/columns of the permutation matrix $B_{\psi(\gamma)}^{00}$ are indexed by the set $\cap_{C \in \psi(\gamma)} \pi^C_0$.  
\end{prop}

To prove Proposition~\ref{prop:block_irreducible_for_exhaustive}, we introduce the following lemmas and definitions. {We start with a lemma to study the properties of the permutation block matrix {$B^{00}_{\psi(\gamma)}$}:}

\begin{lem} \label{lem:two_factor_perm}
If $C = A B$ where $C \in S_n$ and $A$,$B$ are stochastic matrices order $n$, then $B \in S_n$ and $A \in S_n$
\end{lem}
\vspace{-5mm}
\begin{proof}
We have $|\det(C)|=1$ since it is a permutation matrix. Thanks to the homomorphism of the determinant, we have the following:
\begin{align*} \\[-1.7em]
\det(C)&=\det(A)\det(B) \\
1&=|\det(A)||\det(B)|.
\end{align*}
This shows that $|\det(A)|$ and $|\det(B)|$ are $1$. It implies that $\sigma(A)$ and $\sigma(B)$ lie on the unit circle. Lemma~\ref{lem:perm_matrix} then implies that $A\in S_n$ and $B \in S_n$. 
\end{proof}

We need the following lemma to study properties of irreducible block matrices in the submatrix $M_{\psi(\gamma)}$.

\begin{lem}\label{lem:primitive_block}
For a nonzero $\operatorname{ord}_w C$, there exists a permutation matrix $P$ such that the matrix $(M_C)^{\operatorname{ord}_w C}\sim_PB$ {(see~\eqref{eqn:factorization_Pc})} where $B$ is a block diagonal matrix with blocks that are primitive.
\end{lem}
\begin{proof}
From Lemma~\ref{lem:irreducible_blocks} and Definition~\ref{def:non-trivial}, we know that the following holds:
\begin{align*}\label{eqn:diagonal_factorization}
 \underbrace{[ w_p , w_M ]}_{\bar{w}} (\bar{P}_C)^{\operatorname{ord}_w C} &= [ w_p , w_M ] \left[\begin{smallmatrix}
B^{00}_C  & 0 \\
0 &  M_C
 \end{smallmatrix}\right]^{\operatorname{ord}_wC} \\
&= [w_p (B^{00}_C)^{\operatorname{ord}_wC}, \underbrace{w_M (M_C)^{\operatorname{ord}_wC}}_{*} ]\\ 
&=  \underbrace{[ w_p , w_M ]}_{\bar{w}}
\end{align*}
up to relabeling. From the definition of Perron vector and $(*)$, we know that $\operatorname{ord}_wC$ is a multiple of the index of imprimivity of all block matrices $B^{jj}_C$ for $j \geq 1$, from which the result follows. 
\end{proof}
For convenience, we denote the matrix $(M_C)^{\operatorname{ord}_w C}$ by $(M_C)^{w}$.
To prove Proposition~\ref{prop:block_irreducible_for_exhaustive}, we need to consider the graph of a matrix. Let $A$ be a matrix. Let $\mathbb{G}_A$ be a directed graph such that the transpose of its adjacency matrix is equal to the matrix obtained by replacing non-zero entries of the matrix $A$ by one. The directed graph $\mathbb{G}_A$ is called {\em the graph of $A$}.

{The {\em period of the $i^{th}$ entry} of a nonnegative matrix $A$ is defined as $\mathcal{\omega}_A(i):=gcd\{ m: [A^m]_{ii} > 0 , m \in \Nset\}$. If $A$ is irreducible, then $\mathcal{\omega}_A(i) = \mathcal{\omega}_A(j), \forall i,j$~\cite{kitchens1997symbolic}. This common value is called the {\em period} of the matrix $A$, denoted by $\mathcal{\omega}^A$.}

For an irreducible matrix, the period of the matrix is equal to the index of imprimivity of the matrix~\cite{horn2012matrix}. Then, one can easily prove the following corollary to Lemma~\ref{lem:primitive_block}:

\begin{lem}\label{lem:union_of_SCC}
The graph of the matrix $(M_C)^{w}$ is the union of the graphs of the submatrices $(B^{jj}_C)^{w}$, each of which is strongly connected with self-arc at every node for $j \geq 1$. 
\end{lem}

We need to introduce the composition of the graph of matrices. Let $\mathbb{G}_A$ and $\mathbb{G}_B$ be two directed graphs with the same node set $V$. The composition of $\mathbb{G}_A$ with $\mathbb{G}_B$, denoted by $\mathbb{G}_B \circ \mathbb{G}_A$, is a digraph with the node set $V$ and {the edge set} defined as follows: $v_iv_j$ is an edge of $\mathbb{G}_B \circ \mathbb{G}_A$ whenever there is a node $v_k$ such that
$v_iv_k$ is an edge of $\mathbb{G}_A$ and $v_kv_j$ is an edge of $\mathbb{G}_B$~\cite{chen2022gossip}. We have the following Lemma based on the composition definition: 

\begin{lem}\label{lem:composition}
For any sequence of stochastic matrices $A_1,A_2,\cdots,A_k$ which are all of the same size, we have that $\mathbb{G}_{A_k\cdots A_2A_1} = \mathbb{G}_{A_k} \circ \cdots \circ \mathbb{G}_{A_2} \circ \mathbb{G}_{A_1}$
\end{lem}

See~\cite[Lem.~5]{Morse_etAl2008Dynamically} for a proof of Lemma~\ref{lem:composition}. One can easily prove the following lemma:  

\begin{lem}\label{lem:self_arc}
If the graphs $\mathbb{G}_A$ and $\mathbb{G}_B$ have self-arcs at every node, then the union of the {edge sets} of $\mathbb{G}_A$ and $\mathbb{G}_B$ is a subset of the edge set of the graph $\mathbb{G}_{AB}$
\end{lem}
{We note that} the condition of Lemma~\ref{lem:self_arc} says that $A$ and $B$ have all non-zero diagonal entries.

Using the preceding lemmas, we prove Proposition~\ref{prop:block_irreducible_for_exhaustive}:
\begin{proof}[Proof of Proposition~\ref{prop:block_irreducible_for_exhaustive}.] 
{The second condition in the definition of $w$-holonomy for $G$ ensures that the orbit sets of distinct cycles intersect trivially. 
It implies that an exhaustive walk $\gamma$ in $D_G(w)$ is a concatenation of  cycles and self-loops in $D_G(w)$; said differently, $\psi(\gamma)$ is a concatenation of cycles in $G$, where each cycle is repeated a number of times equal to its $w$-order.

Hence, to prove the Proposition, it suffices  to consider  $\gamma:=e_{C_a}^{1}e_{C_a}^{2}\cdots e_{C_a}^{k_a-1}e_{C_b}^{1}\cdots e_{C_b}^{k_b} \cdots e_{C_a}^{k_a-1}$, for which  
$$
\psi(\gamma)=\underbrace{C_a \cdots C_a}_{\operatorname{ord}_w C_a} \cdots \underbrace{C_b \cdots C_b}_{\operatorname{ord}_w C_b} \underbrace{C_a \cdots C_a}_{\operatorname{ord}_w C_a}.  
$$
For convenience, we denote the $\operatorname{ord}_w C$ power of the matrix $P_{C}$ and its principal blocks $(B_{C}^{jj})$ by $P_{C}^{w}$ and $(B_{C}^{jj})^w$, respectively.} Then 
\begin{equation}\label{eqn:p_w_sequence}
     P_{\psi(\gamma)} = P^w_{C_a}P^w_{C_b}\cdots  P^w_{C_a}.
\end{equation}
Recall that $\pi^{C_a}$ and $\pi^{C_b}$ are the partition of the index sets induced by the cycles $C_a$ and $C_b$, respectively. If the partitions $\pi^{C_a}$ and $\pi^{C_b}$ are the same, the Proposition trivially holds and $\pi^{\psi(\gamma)}=\pi^{C_a}$ owing to  Lemma~\ref{lem:primitive_block} and~\eqref{eqn:p_w_sequence}.

We now consider the case whereby $\pi^{C_a}$ and $\pi^{C_b}$ are not equal to each other. We  treat the permutation and irreducible parts of $P_{\psi(\gamma)}$ separately.

{{\em The principal block matrix ${B}^{00}_{\psi(\gamma)}$:}} First consider the permutation block in each matrix $P^w_{C_a}$ and $P^w_{C_b}$. 
{Choose  $i \in \pi^{C_a}_{0}$ and $i \notin \pi^{C_b}_{0}$; from Lemma~\ref{lem:two_factor_perm}, the permutation index set for the matrix $P^w_{C_a}P^w_{C_b}$ does not contain the index $i$.} Owing to the above, the corresponding matrix $P_{\psi(\gamma)}$ only has permutation matrices corresponding to columns/rows indexed by the intersection of the permutation index sets of the  cycles visited by $\psi(\gamma)$, which is $\pi^{\psi(\gamma)}_0 =  \cap_{C \in \psi(\gamma)} \pi^C_0$. {It proves that the rows/columns of the principal matrix ${B}^{00}_{\psi(\gamma)}$ are indexed by $\cap_{C \in \psi(\gamma)} \pi^C_0$ (which is the second part of the Proposition). }

{
{\em The principal block matrix ${B}^{jj}_{\psi(\gamma)}$ for $j\geq1$:}
From Definition~\ref{def:non-trivial}, we have:  
$$
w P^w_{C_a} = w \quad \text{and} \quad w P^w_{C_b} = w.
$$  
By substituting the first equality into the second, we obtain:  
\begin{align}\label{eqn:product}
    w P^w_{C_a}P^w_{C_b} = w.
\end{align}
Now, from Lemma~\ref{lem:two_factor_perm}, we know that the only permutation matrix in the product $P^w_{C_a}P^w_{C_b}$ is the principal block indexed by $\pi^{C_a}_0 \cap \pi^{C_b}_0$. We denote by $w_{\phi}$ the entries of the weight vector $w$ that are indexed by the set $\phi \subseteq \pi$. Since $w$ is invariant under $P^w_{C_a}P^w_{C_b}$ (see \eqref{eqn:product}), the subvector $w_{\pi^{C_a}_0 \cap \pi^{C_b}_0}$ must be a fixed point of the permutation block in $P^w_{C_a}P^w_{C_b}$. Note that some entries of the subvector $w_{\{\pi^{C_a}_0 \cap \pi^{C_b}_0\}}$ might not be equal, and hence  the permutation block in $P^w_{C_a}P^w_{C_b}$ is not necessarily the identity matrix.}

{Now, we can discuss the sets $\pi^{C_a}_0 \setminus \pi^{C_b}_0$ and $\pi^{C_b}_0 \setminus \pi^{C_a}_0$.  From Definition~\ref{def:non-trivial} and Lemma~\ref{lem:irreducible_blocks}, we have:  
\begin{align}
w_{\pi^{C_a}_0} (B^{00}_{C_a})^w = w_{\pi^{C_a}_0}, \quad w_{\pi^{C_b}_0} (B^{00}_{C_b})^w = w_{\pi^{C_b}_0}.  
\end{align}
In words, the subvectors $w_{\pi^{C_a}_0}$ and $w_{\pi^{C_b}_0}$ are a fixed point of the permutation matrices $(B^{00}_{C_a})^w$ and  $(B^{00}_{C_b})^w$, respectively. However, from Lemma~\ref{lem:two_factor_perm}, we know that the {\em only} principal permutation block in the product $P^w_{C_a}P^w_{C_b}$ is indexed by $\pi^{C_a}_0 \cap \pi^{C_b}_0$. This implies that, outside of this intersection, the remaining principal permutation blocks in \( P^w_{C_a} \) and \( P^w_{C_b} \), indexed by \( \pi^{C_a}_0 \setminus \pi^{C_b}_0 \) and \( \pi^{C_b}_0 \setminus \pi^{C_a}_0 \), respectively, are identity matrices. Since these indices correspond to identity matrices, the nodes indexed by $\pi^{C_a}_0 \setminus \pi^{C_b}_0$ and $\pi^{C_b}_0 \setminus \pi^{C_a}_0$ have self-arcs in the graphs $\mathbb{G}_{P_{C_a}^{w}}$ and $\mathbb{G}_{P_{C_b}^{w}}$, respectively. From Lemma~\ref{lem:union_of_SCC}, the nodes in $\mathbb{G}_{P_{C_a}^{w}}$ and $\mathbb{G}_{P_{C_b}^{w}}$ indexed by $\pi^{C_a}_j$ and $\pi^{C_b}_j$ for $j \geq 1$, respectively, form strongly connected components and also have self-arcs.}

{Now, we are in a position to discuss the graph of the product of matrices ${P_{C_a}^{w}}$ and ${P_{C_b}^{w}}$, denoted by $\mathbb{G}_{P_{C_a}^{w}P_{C_b}^{w}}$. From Lemma~\ref{lem:composition}, the graph $\mathbb{G}_{P_{C_a}^{w}P_{C_b}^{w}}$ is the composition of $\mathbb{G}_{P_{C_a}^{w}}$ and $\mathbb{G}_{P_{C_b}^{w}}$. Owing to the discussion above, each node indexed by the set $\pi \setminus (\pi^{C_a}_0 \cap \pi^{C_b}_0)$ has a self-arc in both $\mathbb{G}_{P_{C_a}^{w}}$ and $\mathbb{G}_{P_{C_b}^{w}}$. This implies that, using Lemma~\ref{lem:self_arc}, the edge set of the subgraph of $\mathbb{G}_{P_{C_a}^{w}P_{C_b}^{w}}$, whose nodes are indexed by $\pi \setminus (\pi^{C_a}_0 \cap \pi^{C_b}_0)$, is a subset of the union of the edge sets of the subgraphs of $\mathbb{G}_{P_{C_a}^{w}}$ and $\mathbb{G}_{P_{C_b}^{w}}$, where the nodes are indexed by the same sets of indices. Furthermore, since $\pi^{C_b}$ and $\pi^{C_a}$ partition the {\em same } index set $ \{1, 2, \dots, nm\} $, every element in $ \pi^{C_a}_0 \setminus \pi^{C_b}_0 $ (resp. $ \pi^{C_b}_0 \setminus \pi^{C_a}_0 $) belongs to some $ \pi^{C_b}_j $ (resp. $ \pi^{C_a}_j $) for $ j \geq 1 $.}

{We can then conclude that a path exists from any node indexed by $\pi^{C_a}_i$ to a node indexed by $\pi^{C_b}_j$ for some $i, j \geq 1$ in the graph $\mathbb{G}_{P_{C_a}^{w}P_{C_b}^{w}}$ whenever $\pi^{C_a}_i \cap \pi^{C_b}_j \neq \emptyset$ (note that $i$ and $j$ start from $1$, not $0$). Since the selection of nodes is arbitrary, the nodes indexed by $\pi^{C_bC_a}_k = \pi^{C_a}_i \cup \pi^{C_b}_j$ for some $k, i, j \geq 1$ form a strongly connected component in $\mathbb{G}_{P_{C_a}^{w}P_{C_b}^{w}}$ whenever $\pi^{C_a}_i \cap \pi^{C_b}_j \neq \emptyset$. It is known that a matrix is irreducible if and only if its associated graph is strongly connected~\cite[Thm. 6.2.44]{horn2012matrix}. It follows that each block $B^{kk}_{\psi(\gamma)}$, indexed by $\pi^{\psi(\gamma)}_k$, is irreducible for $k \geq 1$. This completes the proof.}  \end{proof}

For this section, we continue with {Example~\ref{exmp:derivation}}. {Towards illustrating the result of Proposition~\ref{prop:block_irreducible_for_exhaustive}, we analyze the partition of the index set induced by an exhaustive walk using this Example.}

\begin{exmp}[Cont.]\label{exmp:intro_partition}
Assume that the number of states for each node $m$ is $3$ in the graph shown in Figure~\ref{fig:butterfly_graph}; then $x(t) \in \mathbb{R}^{21}$. Consider the cycles $C_1$ and $C_2$. Assume that we have the following partition of the index set: 
\begin{align*}
   \pi^{C_1}& := \{\underbrace{\{ 2,4,7,10,11,\cdots,20,21 \}}_{\pi^{C_1}_{0}},  \underbrace{\{1,3,5\}}_{\pi^{C_1}_1} , \underbrace{\{6,8,9\}}_{\pi^{C_1}_2}   \} \\
   \pi^{C_2} &:=  \{   \underbrace{\{ 4,5,6,7,8,10,11,16,17,\cdots,20,21\} }_{\pi^{C_2}_{0}},\\ & \underbrace{\{12,13\}}_{\pi^{C_2}_1},\underbrace{\{1,2\}}_{\pi^{C_2}_2},\underbrace{\{14,15\}}_{\pi^{C_2}_3}  \}.
\end{align*}
Consider the partition of the index set induced by $C_2C_1$. We observe that the sets $\pi^{C_1}_1$ and $\pi^{C_2}_2$ have a nonempty intersection. Hence, the set $\pi^{C_2C_1}_1:=\pi^{C_1}_1  \cup \pi^{C_2}_2$ is an element of the partition $\pi^{C_2C_1}$. On the other hand, the index sets $\pi^{C_1}_0$ and $\pi^{C_2}_0$ labels the maximal permutation block in the corresponding matrix. From Proposition~\ref{prop:block_irreducible_for_exhaustive}, the set $\pi^{C_2C_1}_0:=\pi^{C_1}_0  \cap \pi^{C_2}_0$ is an element of the partition of the index set $\pi^{C_2C_1}$. Then, we obtain the following:
\begin{align*}
  \pi^{C_2C_1}  &= \{  \underbrace{\{ 4,7,10,11,16,17,\cdots,20,21\}}_{\pi^{C_2C_1}_{0}} , \underbrace{\{1,2,3,5\}}_{\pi^{C_2C_1}_1} ,\\ & \underbrace{\{6,8,9\}}_{\pi^{C_2C_1}_2} , \underbrace{\{12,13\}}_{\pi^{C_2C_1}_3}, \underbrace{\{14,15\}}_{\pi^{C_2C_1}_4}  \} 
\end{align*} 
\end{exmp}

{
The key observations are as follows:

\begin{itemize}
\item For permutation blocks, the index set $\pi_0^{C_2C_1}$ is the intersection of $\pi^{C_2}_0$ and $\pi^{C_1}_0$. Since intersection is a commutative operation, the order in which the cycles appear does not affect the set $\pi_0^{C_2C_1}$.
\item For non-permutation blocks, the index set $\pi_k^{C_2C_1}$ is formed by taking the union of overlapping elements of the partitions induced by $C_1$ and $C_2$ (i.e. $\pi_k^{C_2C_1}=\pi_j^{C_2} \cup \pi_i^{C_1}$ if  $\pi_j^{C_2} \cap \pi_i^{C_1}\neq \emptyset$ for some $k,i,j\geq1$). Here, intersection acts as the selection condition, while union serves as the operation to construct the set $\pi_k^{C_2C_1}$. Since the operations of intersection and union are both commutative, the order of cycle appearance does not affect the resulting partition structure.
\end{itemize}}

This implies that the order of appearance of the cycles in the exhaustive walk does not affect the corresponding index sets. We then conclude that the block structures of the matrices $P_{\psi(\gamma_1)}$ and $P_{\psi(\gamma_2)}$ are the same up to relabeling. We can thus denote the elements of the corresponding partition by $\pi^G$ (e.g., $\pi^G = \pi^{\psi(\gamma)}$). On the other hand, the order in which cycles are traversed affects the order in which local stochastic matrices (which are not necessarily commutative in our work) are multiplied, and thus the block matrices are not necessarily equal.

\subsection{Proof of Theorem~\ref{thm:consensus}}

{The \textit{support of a matrix} $A = [a_{ij}]$, denoted by $supp(A)$, is the set of indices $ij$ such that $a_{ij} \neq 0$. We denote by $\min A$ the smallest non zero entry of $A$: $$\min A = \min_{ij \in supp(A)} a_{ij}.$$}
For a cycle $C$ in $G$ with non-zero order, we define,
\begin{equation}\label{eqn:epsilon_for_cycle}
    \epsilon_C := \min_{1 \leq j \leq m} ( \min B^{jj}_C) 
\end{equation}
where $B^{jj}_C$ is the irreducible block in the matrix $M_C$~\eqref{eqn:factorization_Pc}. We then set:
\begin{equation}\label{eqn:epsilon_for_graph}
    \epsilon := \min_{C \in G} \epsilon_C.
\end{equation}
The coefficient of ergodicity of a stochastic matrix $A\in \mathbb{R}^{n \times n} $ is~\cite{seneta2006non}
\begin{equation}\label{eqn:mu_A}
    \mu(A) :=  \frac{1}{2} \max_{i,j} \sum^{n}{|a_{ik}-a_{jk}|}.
\end{equation}
It is clear that $\mu(A) \leq 1$ for any stochastic matrix $A$. {On the space of $n \times m$ real matrices, we define the following semi-norm for a given $A \in \Rset^{n\times m}$,
$$
    \left\|A\right\|_S := \max_{1\leq j \leq m} \max_{ 1 \leq i_1,i_2 \leq n  } | a_{i_1 j} - a_{i_2 j} |. 
$$
It should be clear that the semi-norm of $A$ is zero if and only if all rows of $A$ are equal.} The following inequality holds for any two stochastic matrices $B$ and $C$~\cite{hajnal1958weak},
\begin{equation}\label{eqn:ergodicity}
 \left\|BC\right\|_S \leq \mu(B) \left\|C\right\|_S.
\end{equation}
For any scrambling matrix $A$, we have the following inequality{~\cite{Morse_etAl2008Dynamically}}: 
\begin{equation}\label{eqn:mu_A_scrumbling}
    \mu(A) \leq 1 - \min(A)   .
\end{equation}
We now need the following lemma: 
\begin{lem}\label{lem:scrambling}
The product of any set of $ l\geq \lfloor\frac{n}{2}\rfloor$ irreducible $n \times n$ stochastic matrices with positive diagonal entries is a scrambling matrix.
\end{lem}
See~\cite[Lem.~5]{chen2022gossip} for a proof of Lemma~\ref{lem:scrambling}.
\begin{defn}[Spanning Sequence]\label{def:spanning_seqeunce}Let $G=(V,E)$ be a simple, undirected graph. A finite sequence of edges of $G$ is spanning if it covers a spanning tree of $G$. An infinite sequence of edges is spanning if it has infinitely many disjoint finite strings that are spanning.
\end{defn}
For our purpose, a random walk in a graph is uniform, i.e., outgoing edges have the same probability of being selected as the next edge.

\begin{lem}\label{lem:exh_walk} Let $G$ be a bridgeless, simple, connected graph. If the set of local stochastic matrices $\{A_{e}\in \mathbb{R}^{nm\times nm}, e \in E\}$ is $w$-holonomic for $G$, then {any random walk $\gamma$ in $D_G(w)$ is an infinite exhaustive walk with probability one} and the sequence of edges $\psi(\gamma)$ is a spanning sequence in $G$.

\end{lem}
\begin{proof} {From Definition~\ref{def:derived_graph_w}, for a cycle $C$, there exists a path from any node in the orbit set $\mathcal{O}_w^C$ to $w$ and from $w$ to any node in the orbit set $\mathcal{O}_w^C$. From Definition~\ref{def:non-trivial}, we know that $w \in \cap_{C \in \vec{\mathcal{C}}} \mathcal{O}_w^C$. This shows that the derived graph $D_w(G)$ is strongly connected; a simple application of the (second) Borel-Cantelli lemma~\cite{hajek2015random} shows that the  event of visiting a node in $D_w(G)$ will occur infinitely often with probability one. This completes the proof for the first claim.}

Now, for the second claim, the set $\mathcal{O}_w^C$ is non-empty for any cycle $C$ by Definition~\ref{def:non-trivial}. Hence, there exists an edge, say $e$, in $D_G(w)$ which has the matrix weight $P_C$. The exhaustive walk $\gamma$ visits all edges in $D_G(w)$. Without loss of generality, assume that $\gamma = \gamma_a \lor e \lor \gamma_b$ where $\gamma_a$ and $\gamma_b$ are walks in $D_G(w)$; we have that $\psi(\gamma) = \psi(\gamma_b) \lor C \lor \psi(\gamma_a)$. This shows that the edges in any cycle $C$ in $G$ are visited. Because  $G$ is bridgeless, every edge in $G$ is covered by at least one cycle. This implies that every edge in $G$ is visited by the sequence of edges $\psi(\gamma)$, which concludes the proof.  
\end{proof}

With the auxiliary results above, we are now in a position to prove Theorem~\ref{thm:consensus}. 
\begin{proof}[Proof of Theorem~\ref{thm:consensus}.] 
Let $w^\top \in \operatorname{int}\Delta^{nm-1}$ be a weight vector. Recall that the set of local stochastic matrices $\{A_{e}\in \mathbb{R}^{nm\times nm}, e \in E\}$ is assumed to be $w$-holonomic for $G$. Let $\gamma$ be an infinite exhaustive walk in $D_G(w)$. 

Proposition~\ref{prop:block_irreducible_for_exhaustive} shows that the block of $P_{\psi(\gamma)}$ indexed by $\pi^{\psi(\gamma)}_0:= \cap_{C \in \psi(\gamma)} \pi_0^C$ is a permutation matrix. Due to Lemma~\ref{lem:exh_walk}, all cycles in $G$ are covered by $\psi(\gamma)$, which implies that $\pi^{\psi(\gamma)}_0 = \cap_{C \in \vec{\mathcal{C}}} \pi_0^C = \pi^G_0$. Consequently, by relabeling the states so that the submatrix indexed by $\pi^G_0$ is in the upper-left corner of $P_{\psi(\gamma)}$, we have proven the first part of assertion~\eqref{itm:block} in Theorem~\ref{thm:consensus}. 

It remains to characterize the limit set $\mathcal{L}$. Let $\Gamma$ be the set of all finite exhaustive walks in $D_G(w)$. Let $\mathcal{S}_{\Gamma}$ be the set of permutation matrices $B^{00}_{\psi(\gamma_i)}$, for all $\gamma_i \in \Gamma$. Owing to the above paragraph, each $B^{00}_{\psi(\gamma_i)}$ is of the same dimension $|\pi^G_0|$. The set $\mathcal{S}_{\Gamma}\subseteq S_{|\pi^G_0|}$ is obviously finite; let $\mathcal{K}$ be the subgroup generated by the elements of $\mathcal{S}_{\Gamma}$. We can write an infinite exhaustive walk $\gamma$ as the concatenation of finite exhaustive closed walks $\gamma_i$, $i\geq 1$, in $D_G(w)$. We then have $\tilde{P}_{\psi(\gamma)}= \cdots  B^{00}_{\psi(\gamma_{i+1})} B^{00}_{\psi(\gamma_i)}B^{00}_{\psi(\gamma_{i-1})} \cdots$ which shows that $\tilde{P}_{\psi(\gamma)} \in \mathcal{K}$.  

We now show that, given a weight vector $w$, the block $M_{\psi(\gamma)}$ in Theorem~\ref{thm:consensus} is {\em uniquely} given. This implies that there exists an injection between the limit set $\mathcal{L}\ni P_{\psi(\gamma)}$ of the process and the set $\mathcal{K}$, which proves that the limit set is finite. It proves the assertion~\eqref{itm:limit_set} in Theorem~\ref{thm:consensus}. To proceed, recall that Proposition~\ref{prop:block_irreducible_for_exhaustive} states that the matrix $M_{\psi(\gamma)}$ consists of principal block matrices that are irreducible with dimensions $|\pi^{G}_i|$ for $i\geq 1$ and are denoted by $M_{\psi(\gamma)}^{ii}$. Let $l_G:= \max_{i}({|\pi^G_i|})$. Let $0:=t_0 < t_1 < t_2 \cdots $ be a monotonically increasing sequence such that every string $\gamma(t_{k} : t_{k+1}) $ for $k \geq 0$, has $\lfloor \frac{l_G}{2} \rfloor$ exhaustive walks in $D_G(w)$. Since $\|M_{\psi(\gamma)}\|_S$ is non-increasing by~\eqref{eqn:ergodicity}, it has a limit for $|\gamma| \to \infty$. We now show that $\|M_{\psi(\gamma)}^{ii}\|_S$ is $0$ for $i \geq 1$.

Lemma~\ref{lem:scrambling} implies that $M_{\psi( \gamma(t_{k} : t_{k+1}) )}^{ii}$  is a scrambling matrix. Plugging the lower bound~\eqref{eqn:epsilon_for_graph} into \eqref{eqn:mu_A_scrumbling}, we have the following inequality $\mu(M_{\psi( \gamma(t_{k} : t_{k+1}) )}^{ii}) \leq (1- \epsilon)$. Then, we can use this inequality in~\eqref{eqn:ergodicity} for each block matrix $M_{\psi(\gamma(t_k:t_{k+1}))}^{ii}$ to obtain:
$$
\begin{aligned}
 \lim_{k \to {\infty } } \left\|M_{\psi(\gamma(0:t_k))}^{ii}\right\|_S & \leq \lim_{k \to \infty}  (1-\epsilon)\left\|M_{\psi(\gamma(0:t_{k-1}))}^{ii}\right\|_S  \\
& \leq(1-\epsilon)^{k} = 0
\end{aligned}
$$
which implies that $\lim_{k \to {\infty } } \left\|M_{\psi( \gamma(0: t_{k}) )}^{ii}\right\|_S = 0$. We conclude using~\cite{seneta2006non} that $M_{\psi( \gamma(0 :t_k) )}^{ii}$ converges to a rank-one matrix, say 
$M_{\psi(\gamma)}^{ii} =  \mathds{1} p_i^\top$ for some vector $p_i \in \mathbb{R}^{|\pi^G_i|}$. This establishes asymptotic convergence in the assertion~\eqref{itm:rank_one} in Theorem~\ref{thm:consensus}. 

It remains to provide an explicit characterization of the vector $p_i^\top$ such that $M_{\psi(\gamma)}^{ii}=\mathds{1}p_i^\top$ as a function of the weight vector $w$. We denote by $\Tilde w$ the entries of the weight vector $w$ which are indexed by the set $\cup_{i \geq 1} {\pi^G_i}$. By construction of the derived graph $D_G(w)$, the walk $\gamma$ induces the mapping:
$$
\tilde{w} \mapsto \tilde{w} M_{\psi(\gamma)} ( =\tilde{w}). 
$$
Then, we have the following up to relabeling,
\begin{equation}\label{eqn:block_for_limit}
   \tilde{w} = \tilde{w} \left[ \begin{smallmatrix}
        M_{\psi(\gamma)}^{11}   &  &  \\
         &  M_{\psi(\gamma)}^{22} & \\
         & & \ddots & 
   \end{smallmatrix} \right] = \tilde{w} \left[ \begin{smallmatrix}
      \mathds{1} p_1^\top   &  &  \\
         &  \mathds{1} p_2^\top & \\
         & & \ddots & 
    \end{smallmatrix} \right].
\end{equation}
 From the block structure \eqref{eqn:block_for_limit}, we have that,
$$
p_i^\top := \frac{ [\Tilde{w}]_{j \in \pi^G_i}   }{\alpha_i} \mbox{ where } \alpha_i:=\sum_{j \in \pi^G_i} \Tilde{w}_j.
$$
Since $w$ is a weight vector by definition, we know that $\alpha_i \in (0,1]$. This then shows that $p_i$ has no zero entry.  
\end{proof}

\section{Summary and outlook}\label{sec:conclusion}

In this paper, we have investigated the weighted average consensus problem for a gossiping network of agents with vector states. We have introduced the concept of $w$-holonomy for a set of stochastic matrices, which helped us to investigate the existence of non-trivial, finite holonomy groups in the gossip process. The allowable sequences of updates in the gossip process were obtained as closed walks in the so-called derived graph $D_G(w)$, in that any infinite exhaustive walks in $D_G(w)$ could be mapped to an allowable sequence of updates for the gossip process. Such sequences could be implemented in a decentralized manner, and we have shown that the corresponding infinite product of stochastic matrices converges to a finite limit set, whose elements we have explicitly characterized.

Our results have established a unified framework that connects the methodologies presented in~\cite{chen2022gossip} and~\cite{belabbas2021triangulated}. Indeed, on the one hand, we have extended the framework of~\cite{chen2022gossip} by allowing gossip processes that display non-trivial holonomy groups, which results in a finite limit set for the process (as stated in Theorem~\ref{thm:consensus}). This is in contrast to~\cite[Thm.~1]{chen2022gossip}, where the limit set is a singleton. As a drawback of the existence of non-trivial holonomy groups, our results require following an allowable sequence of updates in the gossip process, whereas the order of the gossiping pairs does not matter in~\cite[Thm.~1]{chen2022gossip}. 

On the other hand, in~\cite{belabbas2021triangulated}, allowable sequences have been found through a derived graph,  (see~\cite[Definition 2.3]{belabbas2021triangulated}) nodes of which correspond to three nodes in the gossip graph that communicate simultaneously. This approach, motivated by the design of secure protocols, results in allowable sequences consisting of an infinite concatenation of triangles within the communication graph. In our work, we have developed a different perspective. Each node in our derived graph corresponds to an element of the orbit sets of the weight vector around a cycle. Consequently, our allowable sequences consist of the concatenation of cycles in the communication graph. It is worth noting that our approach requires a bridgeless communication graph, while the methodology presented by~\cite{belabbas2021triangulated} requires a triangulated Laman graph.

The present work can be extended in several directions. Among others, we focus here on algorithm design, which comprises two parts. The first is to characterize the set of local stochastic matrices that yield a predefined consensus weight for agents with vector-valued states. It is relatively straightforward to use the results developed in this paper to develop a method that yields the desired local stochastic matrices given that the communication graph is {\em bridgeless}. Removing this topological constraint requires further research. This leads us to the second aspect, which is to remove the requirement of bridgeless graphs. To understand what it entails, we first note that the requirement can be traced back to the definition of $w$-holonomy involving cycles in $G$. Thus, one approach to remove the requirement is to modify the definition of $w$-holonomy to consider paths rather than cycles in the communication graph.

Namely, for a path $\zeta$, we would redefine the orbit set in~\eqref{eqn:orbit_set} as $\mathcal{O}_w^{\zeta}:=\{ w_{\zeta}^{(a)} \in \mathbb{R}^{nm} | w_{\zeta}^{(a)} = w({P}_{\zeta})^a \mbox{ for }  a \in \mathbb{N} \}$ and state that if the set $\mathcal{O}_w^{\zeta}$ has finite and non-zero cardinality, then set of local stochastic matrices $\{ A_e \in \mathbb{R}^{nm\times nm},\forall e \in \zeta\}$ will be $w$-holonomic for $\zeta$. With this modification, we can still employ the derived graph approach to characterize the allowable sequences. However, this introduces a significant challenge: the allowable sequences of updates cannot be followed in a decentralized manner (at least in an obvious manner).
The ability of gossip processes to operate without a central authority is however crucial. This highlights the need to further understand the connection between topological constraints on $G$, and the development of decentralized update rules.

\bibliographystyle{ieeetr}
\bibliography{sample}

\begin{thebibliography}{10}

\bibitem{degroot1974reaching}
M.~H. DeGroot, ``Reaching a consensus,'' {\em Journal of the American Statistical Association}, vol.~69, no.~345, pp.~118--121, 1974.

\bibitem{boyd2006randomized}
S.~Boyd, A.~Ghosh, B.~Prabhakar, and D.~Shah, ``Randomized gossip algorithms,'' {\em IEEE \protect{Transactions on Information Theory}}, vol.~52, no.~6, pp.~2508--2530, 2006.

\bibitem{benezit2010weighted}
F.~B{\'e}n{\'e}zit, V.~Blondel, P.~Thiran, J.~Tsitsiklis, and M.~Vetterli, ``Weighted gossip: Distributed averaging using non-doubly stochastic matrices,'' in {\em 2010 IEEE ISIT}, pp.~1753--1757, IEEE, 2010.

\bibitem{belabbas2021triangulated}
M.-A. Belabbas and X.~Chen, ``Triangulated \protect{L}aman graphs, local stochastic matrices, and limits of their products,'' {\em Linear Algebra and its Applications}, vol.~619, pp.~176--209, 2021.

\bibitem{chen2022gossip}
X.~Chen, M.-A. Belabbas, and J.~Liu, ``Gossip over holonomic graphs,'' {\em Automatica}, vol.~136, p.~110088, 2022.

\bibitem{wolf2002dif}
W.~Kuhnel, {\em Differential Geometry : Curves - Surfaces - Manifolds}.
\newblock Providence, R.I: American Mathematical Society, 2002.

\bibitem{bolouki2015consensus}
S.~Bolouki and R.~P. Malham{\'e}, ``Consensus algorithms and the decomposition-separation theorem,'' {\em IEEE \protect{Transactions} on Automatic Control}, vol.~61, no.~9, pp.~2357--2369, 2015.

\bibitem{bolouki2013ergodicity}
S.~Bolouki and R.~P. Malham{\'e}, ``Ergodicity and class-ergodicity of balanced asymmetric stochastic chains,'' in {\em 2013 European Control Conference (ECC)}, pp.~221--226, IEEE, 2013.

\bibitem{touri2012approximations}
B.~Touri and A.~Nedic, ``On approximations and ergodicity classes in random chains,'' {\em IEEE \protect{Transactions} on Automatic Control}, vol.~57, no.~11, pp.~2718--2730, 2012.

\bibitem{zhu_etal2021federated}
Z.~Zhu, J.~Zhu, J.~Liu, and Y.~Liu, ``Federated bandit: A gossiping approach,'' {\em Proc. ACM Meas. Anal. Comput. Syst.}, vol.~5, no.~1, 2021.

\bibitem{olfati2007distributed}
R.~Olfati-Saber, ``Distributed kalman filtering for sensor networks,'' in {\em 2007 46th IEEE Conference on Decision and Control}, pp.~5492--5498, IEEE, 2007.

\bibitem{chen_etal2017cluster}
X.~Chen, M.-A. Belabbas, and T.~Ba\c{s}ar, ``Cluster consensus with point group symmetries,'' {\em SIAM Journal on Control and Optimization}, vol.~55, no.~6, pp.~3869--3889, 2017.

\bibitem{sundaram2018distributed}
S.~Sundaram and B.~Gharesifard, ``Distributed optimization under adversarial nodes,'' {\em IEEE \protect{Transactions} on Automatic Control}, vol.~64, no.~3, pp.~1063--1076, 2018.

\bibitem{touri2015continuous}
B.~Touri and B.~Gharesifard, ``Continuous-time distributed convex optimization on time-varying directed networks,'' in {\em 2015 54th IEEE Conference on Decision and Control (CDC)}, pp.~724--729, IEEE, 2015.

\bibitem{bayram2024age}
E.~Bayram, M.~Bastopcu, M.-A. Belabbas, and T.~Ba{\c{s}}ar, ``Age of k-out-of-n systems on a gossip network,'' in {\em 2024 58th Asilomar Conference on Signals, Systems, and Computers}, pp.~1807--1811, IEEE, 2024.

\bibitem{bayram2024ageCoded}
E.~Bayram, M.~Bastopcu, M.-A. Belabbas, and T.~Ba{\c{s}}ar, ``Age of coded updates in gossip networks under memory and memoryless schemes,'' {\em IEEE Transactions on Communications}, vol.~73, no.~12, pp.~14562--14578, 2025.

\bibitem{bolouki2015eminence}
S.~Bolouki, R.~P. Malham{\'e}, M.~Siami, and N.~Motee, ``{\'E}minence grise coalitions: On the shaping of public opinion,'' {\em IEEE \protect{Transactions} on Control of Network Systems}, vol.~4, no.~2, pp.~133--145, 2015.

\bibitem{zhu2010discrete}
M.~Zhu and S.~Martinez, ``Discrete-time dynamic average consensus,'' {\em Automatica}, vol.~46, no.~2, pp.~322--329, 2010.

\bibitem{blondel2014decide}
V.~D. Blondel and A.~Olshevsky, ``How to decide consensus? a combinatorial necessary and sufficient condition and a proof that consensus is decidable but np-hard,'' {\em SIAM Journal on Control and Optimization}, vol.~52, no.~5, pp.~2707--2726, 2014.

\bibitem{nedic2016convergence}
A.~Nedi{\'c} and J.~Liu, ``On convergence rate of weighted-averaging dynamics for consensus problems,'' {\em IEEE \protect{Transactions} on Automatic Control}, vol.~62, no.~2, pp.~766--781, 2016.

\bibitem{tsitsiklis1984problems}
J.~N. Tsitsiklis, {\em Problems in Decentralized Decision Making and Computation}.
\newblock PhD thesis, Massachusetts Institute of Technology, 1984.

\bibitem{kashyap2007quantized}
A.~Kashyap, T.~Ba{\c{s}}ar, and R.~Srikant, ``Quantized consensus,'' {\em Automatica}, vol.~43, no.~7, pp.~1192--1203, 2007.

\bibitem{etesami2015convergence}
S.~R. Etesami and T.~Ba{\c{s}}ar, ``Convergence time for unbiased quantized consensus over static and dynamic networks,'' {\em IEEE \protect{Transactions} on Automatic Control}, vol.~61, no.~2, pp.~443--455, 2015.

\bibitem{chen2016distributed}
X.~Chen, M.-A. Belabbas, and T.~Ba{\c{s}}ar, ``Distributed averaging with linear objective maps,'' {\em Automatica}, vol.~70, pp.~179--188, 2016.

\bibitem{li2021event}
X.~Li and J.~Zhu, ``Event-triggered weighted average consensus in networks of dynamic agents with time-varying delay,'' {\em IFAC-PapersOnLine}, vol.~54, no.~18, pp.~127--132, 2021.

\bibitem{Morse_etal2008ReachingConsensus}
M.~Cao, A.~S. Morse, and B.~Anderson, ``Reaching a consensus in a dynamically changing environment: Convergence rates, measurement delays, and asynchronous events,'' {\em SIAM Journal on Control and Optimization}, vol.~47, no.~2, pp.~601--623, 2008.

\bibitem{fang2005asynchronous}
L.~Fang, P.~J. Antsaklis, and A.~Tzimas, ``Asynchronous consensus protocols: Preliminary results, simulations and open questions,'' in {\em Proceedings of the 44th IEEE Conference on Decision and Control}, pp.~2194--2199, IEEE, 2005.

\bibitem{Morse_etAl2008Dynamically}
M.~Cao, A.~S. Morse, and B.~Anderson, ``Reaching a consensus in a dynamically changing environment: A graphical approach,'' {\em SIAM Journal on Control and Optimization}, vol.~47, no.~2, pp.~575--600, 2008.

\bibitem{ren2005consensus}
W.~Ren and R.~W. Beard, ``Consensus seeking in multiagent systems under dynamically changing interaction topologies,'' {\em IEEE \protect{Transactions} on Automatic Control}, vol.~50, no.~5, pp.~655--661, 2005.

\bibitem{hendrickx2011new}
J.~M. Hendrickx and J.~Tsitsiklis, ``A new condition for convergence in continuous-time consensus seeking systems,'' in {\em 2011 50th IEEE Conference on Decision and Control and European Control Conference}, pp.~5070--5075, IEEE, 2011.

\bibitem{hendrickx2012convergence}
J.~M. Hendrickx and J.~Tsitsiklis, ``Convergence of type-symmetric and cut-balanced consensus seeking systems,'' {\em IEEE \protect{Transactions} on Automatic Control}, vol.~58, no.~1, pp.~214--218, 2012.

\bibitem{fagnani2008randomized}
F.~Fagnani and S.~Zampieri, ``Randomized consensus algorithms over large scale networks,'' {\em IEEE Journal on Selected Areas in Communications}, vol.~26, no.~4, pp.~634--649, 2008.

\bibitem{touri2012backward}
B.~Touri and A.~Nedi{\'c}, ``On backward product of stochastic matrices,'' {\em Automatica}, vol.~48, no.~8, pp.~1477--1488, 2012.

\bibitem{touri2010ergodicity}
B.~Touri and A.~Nedic, ``On ergodicity, infinite flow, and consensus in random models,'' {\em IEEE \protect{Transactions} on Automatic Control}, vol.~56, no.~7, pp.~1593--1605, 2010.

\bibitem{aghajan2021ergodicity}
A.~Aghajan and B.~Touri, ``Ergodicity of continuous-time distributed averaging dynamics: A spanning directed rooted tree approach,'' {\em IEEE \protect{Transactions} on Automatic Control}, vol.~67, no.~2, pp.~918--925, 2021.

\bibitem{liu2011deterministic}
J.~Liu, S.~Mou, A.~S. Morse, B.~Anderson, and C.~Yu, ``Deterministic gossiping,'' {\em Proceedings of the IEEE}, vol.~99, no.~9, pp.~1505--1524, 2011.

\bibitem{he2011periodic}
F.~He, A.~Morse, J.~Liu, and S.~Mou, ``Periodic gossiping,'' {\em IFAC Proceedings Volumes}, vol.~44, no.~1, pp.~8718--8723, 2011.

\bibitem{bayram2025construct}
E.~Bayram and M.-A. Belabbas, ``Constructing stochastic matrices for weighted averaging in gossip networks,'' {\em IFAC-PapersOnLine}, vol.~59, no.~4, pp.~85--90, 2025.
\newblock 10th IFAC Conference on Networked Systems NECSYS 2025.

\bibitem{merris2011graph}
R.~Merris, {\em Graph Theory}.
\newblock John Wiley \& Sons, 2011.

\bibitem{meyer2000matrix}
C.~D. Meyer, {\em Matrix analysis and applied linear algebra}.
\newblock OT ; 71, Philadelphia: Society for Industrial and Applied Mathematics, 2000.

\bibitem{teaching_tip}
J.~Ding and N.~H. Rhee, ``Teaching tip: When a matrix and its inverse are stochastic,'' {\em The College Mathematics Journal}, vol.~44, no.~2, pp.~108--109, 2013.

\bibitem{kitchens1997symbolic}
B.~P. Kitchens, {\em Symbolic Dynamics: One-sided, Two-sided and Countable State Markov Shifts}.
\newblock Springer, 1997.

\bibitem{horn2012matrix}
R.~A. Horn and C.~R. Johnson, {\em Matrix Analysis}.
\newblock Cambridge University Press, 2012.

\bibitem{seneta2006non}
E.~Seneta, {\em Non-negative Matrices and Markov Chains}.
\newblock Springer, 2006.

\bibitem{hajnal1958weak}
J.~Hajnal and M.~S. Bartlett, ``Weak ergodicity in non-homogeneous markov chains,'' in {\em Mathematical Proceedings of the Cambridge Philosophical Society}, vol.~54, pp.~233--246, Cambridge University Press, 1958.

\bibitem{hajek2015random}
B.~Hajek, {\em Random Processes for Engineers}.
\newblock Cambridge University Press, 2015.

\end{thebibliography}

\end{document}